\documentclass[11pt,reqno,twoside]{article}

\usepackage{empheq}
\usepackage{graphicx}
\usepackage{float}
\usepackage{subfigure}
\usepackage{amssymb}
\usepackage{epstopdf}
\usepackage[asymmetric,top=3.5cm,bottom=4.3cm,left=3.1cm,right=3.1cm]{geometry}
\usepackage{hyperref}
\usepackage{bm}
\usepackage{amsmath,amsfonts,amsthm,mathrsfs,amssymb,cite}
\usepackage[usenames]{color}
\usepackage{array} 
\usepackage{paralist} 
\usepackage{verbatim} 
\usepackage{tabularx}
\usepackage{fancyhdr}
\usepackage{graphicx, subfigure,threeparttable,tikz}
\usetikzlibrary {arrows.meta}

\newtheorem{theorem}{Theorem}[section]
\newtheorem{corollary}[theorem]{Corollary}
\newtheorem{lemma}{Lemma}[section]
\newtheorem{proposition}{Proposition}[section]

\theoremstyle{definition}
\newtheorem{definition}{Definition}[section]
\theoremstyle{remark}

\newtheorem{remark}{Remark}[section]
\numberwithin{equation}{section}

\newcommand{\bbr}{\mathbb{R}}
\newcommand{\bbn}{\mathbb{N}}

\newcommand{\mi}{\mathrm{i}}

\newcommand{\bx}{\mathbf{x}}
\newcommand{\by}{\mathbf{y}}
\newcommand{\bz}{\mathbf{z}}
\newcommand{\bp}{\mathbf{p}}
\newcommand{\bc}{\mathbf{c}}

\newcommand{\ocal}{\mathcal{O}}

\newcommand{\abs}[1]{\left\vert#1\right\vert}
\newcommand{\coma}{\; , \;}
\newcommand{\dx}[1]{\, \mathrm{d} #1}
\newcommand{\tu}{\tilde{u}}

\title{Mathematical analysis of transverse EM field concentration for adjacent obstacles with nonlocal boundary conditions in the quasi‑static regime}

\begin{document}

\author{
		Yueguang Hu\footnote{Yau Mathematical Sciences Center, Tsinghua University, Beijing, China. (yueghu2@gmail.com; yghumath@tsinghua.edu.cn;).}
		\and
		Hongjie Li\footnote{Yau Mathematical Sciences Center, Tsinghua University, Beijing, China. The work of this author was substantially supported by NSFC grant (12401561). (hongjieli@tsinghua.edu.cn; hongjie\_li@yeah.net).}
		\and
		Hongyu Liu\footnote{Department of Mathematics, City University of Hong Kong, Kowloon, Hong Kong, China. The work of this author was supported by NSFC/RGC Joint Research Scheme, N\_CityU101/21, ANR/RGC Joint Research Scheme, A-CityU203/19, and the Hong Kong RGC General Research Funds (projects 11311122,  11304224, and 11303125).
			 ( hongyliu@cityu.edu.hk; hongyu.liuip@gmail.com).}
	}

\date{}
\maketitle

\begin{abstract}

This paper presents a rigorous mathematical analysis of transverse electromagnetic (EM) field concentration between two adjacent obstacles within the framework of the quasi-static approximation. We investigate three degenerate conductivity models recently introduced in \cite{Hu2025}, two of these incorporating  nonlocal boundary conditions to capture fundamental physical phenomena, such as surface nonlocality and thin-layer interactions. Our primary results establish sharp conditions for gradient blowup and derive the corresponding optimal blowup rates. These findings elucidate how nonlocal boundary conditions modify classical gradient estimates. Furthermore, we analyze the influence of wave frequency, demonstrating that it mitigates the severity of field concentration even in the limit of a vanishing gap distance. Consequently, this work extends the classical theory of field enhancement in plasmonic and metamaterial systems to incorporate nonlocal surface effects, yielding precise asymptotic formulas that are essential for the quantitative design of nanophotonic devices.

\medskip 


\end{abstract}

\section{Introduction and problem formulation}
\subsection{Mathematical setting and main results}

This paper aims to establish optimal estimates for transverse electromagnetic field concentration between two adjacent obstacles within the quasi-static regime. We begin by presenting the mathematical formulation, followed by a discussion of our main analytical results. 

Let $D:= D_1 \cup D_2\subset \bbr^2$ represent the cross-section of two infinitely long cylindrical inclusions. We consider time-harmonic wave scattering in a homogeneous medium containing these two adjacent inclusions. 
When the inclusion $D$ is characterized 
by its electric permittivity and magnetic permeability in transverse electromagnetic scattering, the associated scalar field $u$ 
satisfies the following Helmholtz system:
\begin{equation}\label{eq:helmhotz}
\begin{cases}
\displaystyle \nabla \cdot \left( \tilde{\varepsilon}^{-1} \chi_{D} +  \varepsilon^{-1} \chi_{\bbr^2 \backslash \overline{D}}\right) \nabla u (\bx)  + \omega^2 \left(\tilde{\mu} \chi_{D} +  \mu \chi_{\bbr^2 \backslash \overline{D}}\right) u(\bx) = 0 , \medskip \\
\displaystyle  \lim_{r \rightarrow \infty} r^{\frac{1}{2}}\left(\frac{\partial}{\partial \nu}-\mi k \right) \left(u(\bx)-u^i(\bx)\right) = 0 \coma r = \abs{\bx}, \medskip
\end{cases}
\end{equation}
where $k:= \omega\sqrt{\varepsilon\mu}$ represents the wavenumber and $\mi : =\sqrt{-1}$ denotes the imaginary unit. Here, $u^i$ represents an incident field satisfying the homogeneous Helmholtz equation $\Delta u^i + k^2 u^i = 0$ throughout the entire space. The final limit represents the Sommerfeld radiation condition, which characterizes the outgoing nature of the scattered field at infinity.

We investigate whether the gradient $\nabla u$ exhibits blowup between the two inclusions as the material parameters degenerate to extreme values. 
As $\tilde{\varepsilon}$ approaches zero, the medium scattering problem \eqref{eq:helmhotz} reduces to an obstacle scattering problem with constant Dirichlet boundary data:
\begin{equation}\label{mainequation}
\begin{cases}
\displaystyle \Delta u + k^2 u= 0  \quad \textrm{in}  \quad  \bbr^2 \backslash \overline{D_1\cup D_2} , \medskip\\
\displaystyle u = \lambda_j \hspace{45pt} \textrm{on} \quad  \partial D_j \coma j=1,2,\medskip\\
\displaystyle\lim_{r \rightarrow \infty} r^{\frac{1}{2}}\left(\frac{\partial}{\partial \nu}-\mi k \right) \left(u-u^i\right) = 0 , \medskip
\end{cases}
\end{equation}
where the Dirichlet values $\lambda_j$ are determined by $\tilde{\mu}$ as follows:
\begin{enumerate}
\item[i).] When $\tilde{\mu} = 0 \coma  \lambda_j$ is determined by the condition:
\begin{equation}\label{infinitekappa}
\int_{\partial D_j} \frac{\partial u}{\partial \nu} \dx{s} = 0.
\end{equation}
\item[ii).] When $\tilde{\mu} \sim 1 \coma \lambda_j$ is given by
\begin{equation}\label{onekappa}
\lambda_j = -  \frac{\tau}{k^2\mathcal{V}(D_j) } \int_{\partial D_j} \frac{\partial u}{\partial \nu}\dx{s},
\end{equation}
where $\mathcal{V}(D_j)$ denotes the Lebesgue measure of the inclusion $D_j$, and $\tau = \mu/\tilde{\mu}$.
\item[iii).] When $\tilde{\mu} = \infty$,  the inclusions $D_j$ reduce to  perfect electric conductors with
\begin{equation}\label{zerokappa}
\lambda_j = 0.
\end{equation}
\end{enumerate}
These three boundary conditions were derived in \cite{Hu2025}, which also establishes the well-posedness of system \eqref{mainequation} under each respective condition.

We first define the geometric setting and symbolic notation with greater precision. 
The spatial variable is denoted by $\mathbf{x} = (x_1, x_2)$, and all vectors and vector fields are written in boldface. Let $\mathbf{c}_j$ and $r_j$ for $j=1,2$ be the centers and radii of the two disks, respectively, and set $2\epsilon := \operatorname{dist}(D_1, D_2)$.
Applying rigid motions if necessary, we position the origin $\mathbf{0}$ at the midpoint between the two disks so that $\bc_1 = (0, r_1+\epsilon) \coma \bc_2 = (0, -r_2-\epsilon)$. This geometric configuration is illustrated in Figure \ref{fig:twodisks}.
Throughout this paper, we adopt a refined scale for the quasi-static regime, defined as follows:
\begin{definition}[Quasi-static ansatz]
The quasi-static regime for the disks $D_1$ and $D_2$ is characterized by the condition: 
$$k \cdot \max \{r_1,r_2, \epsilon\} \ll 1. $$
\end{definition}
We now present the principal results of this study. 
\begin{theorem}\label{thm:blowup}
Let $u$ be the unique solution to \eqref{mainequation} subject to the boundary condition \eqref{infinitekappa} or \eqref{onekappa} within the quasi-static regime.  If $\min \{r_1,r_2\}\gg \epsilon$, then
\begin{equation*}
\lambda_2 -\lambda_1 = 4\sqrt{\frac{r_1r_2}{r_1+r_2}}\sqrt{\epsilon} \Biggl(\partial_{x_2} u^i(\mathbf{0}) + \ocal \biggl(\epsilon + \partial_{x_2} u^i(\mathbf{0}) k^2\max \left\{r_1^2,r_2^2\right\}\sqrt{\frac{\epsilon}{\min\{r_1,r_2\}}} \biggr) \Biggr) .
\end{equation*}
Applying the mean value theorem, there exists a point $\mathbf{x_0}$ between $D_1$ and $D_2$ such that
\begin{align*}
\abs{\nabla u(\mathbf{x_0})} \geqslant  2\sqrt{\frac{r_1r_2}{(r_1+r_2)\epsilon}} \Biggl(\abs{\partial_{x_2} u^i(\mathbf{0})} + \ocal \biggl(\epsilon +  k^2\max \left\{r_1^2,r_2^2\right\}\sqrt{\frac{\epsilon}{\min\{r_1,r_2\}}} \biggr) \Biggr) .
\end{align*}
Moreover, the aforementioned lower bound is optimal in the sense that the following upper bound holds: 
\begin{align*}
\| \nabla u(\bx)\|_{(L^\infty (\bbr^2 \backslash \overline{D_1\cup D_2}) }\leqslant  2\sqrt{\frac{r_1r_2}{(r_1+r_2)\epsilon}} \Biggl( \abs{\partial_{x_2} u^i(\mathbf{0})} + \ocal \biggl(\epsilon + \partial_{x_2} u^i(\mathbf{0}) k^2\max \left\{r_1^2,r_2^2\right\}\\
\sqrt{\frac{\epsilon}{\min\{r_1,r_2\}}}\biggr) \Biggr) + C\left(1+\abs{\lambda_1} + \abs{\lambda_2}\right).
\end{align*}
where $C$ is a constant independent of $r_1,r_2,\epsilon$, and $k$.
\end{theorem}

\begin{remark}[Frequency mitigation]
The incident field can be expanded into a series involving Bessel functions:
$u^i = \sum_{n \in \bbn} a_n J_n(kr)e^{\mi n \theta}, a_n \in \mathbb{C}.$
This implies that the optimal blowup order is $a_1 k \sqrt{\min \{r_1,r_2\}/\epsilon}$.
If we further assume that $\min \{r_1,r_2\} \sim 1, k \ll 1,$ and $\epsilon \ll 1$ within the quasi-static regime, then
\begin{equation}\label{eq:optiterm}
    \| \nabla u(\bx)\|_{(L^\infty (\bbr^2 \backslash \overline{D_1\cup D_2}) } \sim  \frac{k}{\sqrt{\epsilon}}.
\end{equation}
Thus, the frequency $k$ mitigates the severity of the gradient blowup, and the gradient may even remain uniformly bounded if $k = \ocal (\sqrt{\epsilon})$. 
\end{remark}

\begin{remark}[Blowup convergence]
Based on the optimal term \eqref{eq:optiterm}, the gradient appears to remain uniformly bounded as $k \rightarrow 0$. This phenomenon appears counter-intuitive, as it does not converge to the static limit ($k = 0$) \cite{lim2009blow,Kang2013}. 
Using the Bessel series representation, we observe that $u^i = a_0 + \ocal (kr)$. Thus, the incident field converges to a constant, under which no gradient blowup occurs in the static case. This convergence can be achieved by selecting a specific incident field, such as $u^i = a/k \sin (k \bx \cdot\mathbf{d}) \coma \mathbf{d} \in \mathbb{S}^1\backslash (\pm 1,0), a \neq 0$.
\end{remark}

\begin{figure}[htbp!]
\centering
\begin{tikzpicture}
\clip (-4,-4.5) rectangle (5.3,5); 
\node at (4.7,-0.3) {$x_{1}$};
\node[black] at (0.4,4.6) {$x_{2}$};

\filldraw[lightgray] (0,2.2) circle (1.8);
\filldraw[lightgray, line width=1pt] (0,-2.2) circle (1.8);
\draw[gray, line width=1pt] (0,2.2) circle (1.8);
\draw[gray, line width=1pt] (0,-2.2) circle (1.8);
\filldraw[black] (0,2.2) circle (0.05);
\filldraw[black] (0,-2.2) circle (0.05);
\draw (0,2.2) -- (1.4,3.3);
\draw (0,-2.2) -- (1.4,-3.3);

\draw [black,thick,dashed] [->] (-4,0) -- (5,0);
\draw [black,thick,dashed] [->] (0,-4.5) -- (0,5);

\node[black] at (0.4,2.1) {$\bc_{1}$};
\node[black] at (0.4,-2.2) {$\bc_{2}$};
\node[black] at (1.2,-2.5) {$r_{2}$};
\node[black] at (1.2,2.5) {$r_{1}$};
\node at (-0.8,2.8) {$ D_1$};
\node at (-0.8,-2.8) {$ D_2$};
\node at (0.25,-0.2) {$\mathbf{0}$};

\node at (0.8,0) {$\boldsymbol{\Big\}}$};
\node at (3.4,0.5) {$2\epsilon := \textrm{dist} \left(D_1, D_2\right)$};

\draw [line width=1.5pt,black, arrows = {-Stealth[black]}] (1,0.1) -- (1.7,0.4);
\end{tikzpicture}
\caption{Geometric illustration of the two disks $D_j$ for $j=1,2$.}
\label{fig:twodisks}
\end{figure}
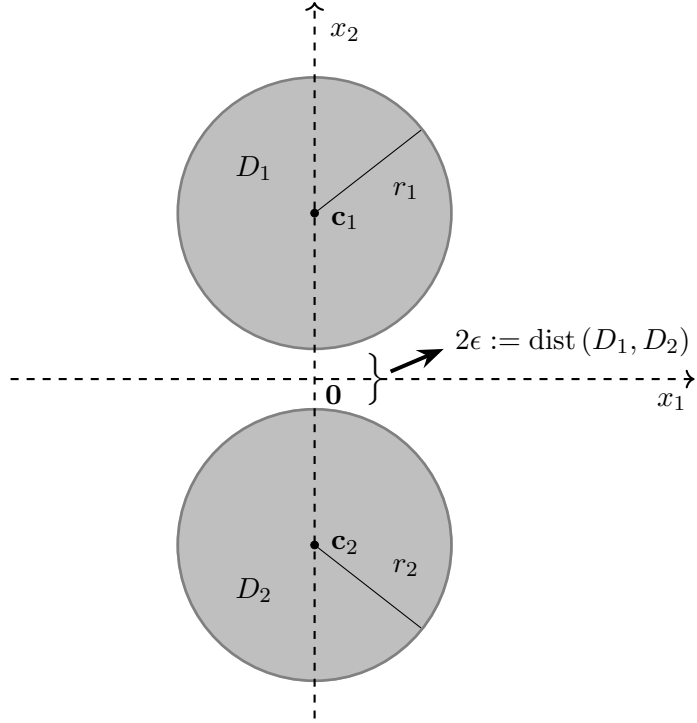

\begin{theorem}\label{thm:bounded}
Let $u$ be the unique solution to \eqref{mainequation} subject to the boundary condition \eqref{infinitekappa} or \eqref{onekappa} within the quasi-static regime. Assume that $a\ln a = \ocal (b)$ whenever $a \ll b$ for any two positive constants $a$ and $b$.  If $\min \{r_1,r_2\} = \ocal (\epsilon)$, then
\begin{equation*}
\lambda_2 -\lambda_1  =  \epsilon \cdot \partial_{x_2}u^i(\bx_*) \; \ocal \left( 1+k^2\max\{r_1,r_2\}\max\{r_1,r_2, \epsilon\} \right),
\end{equation*}
where $\bx_*$ represents a point located on the line segment connecting $\bc_2$ and $\bc_1$.
Furthermore, the gradient $\|\nabla u\|$ remains uniformly bounded outside the inclusions $D_1$ and $D_2$.
\end{theorem}

\begin{remark}[Blowup condition]
Combining Theorem \ref{thm:blowup} and Theorem \ref{thm:bounded}, the sharpest condition for gradient blowup is determined to be 
$
\min \{r_1,r_2\} \gg \epsilon \; \mathrm{and} \;   \partial_{x_2}u^i(\mathbf{0}) \neq 0. 
$
Taking frequency mitigation into account, this blowup condition must be modified to
\begin{equation*}
\min \{r_1,r_2\} \gg \epsilon   \quad \mathrm{and} \quad k \gg \sqrt {\frac{\epsilon}{\min \{r_1,r_2\}}}.
\end{equation*}
\end{remark}

The result for the perfect electric conductors follows immediately as a corollary to Theorem \ref{thm:blowup}.
\begin{corollary}
Let $u$ be the solution to \eqref{mainequation} subject to the perfect electric boundary condition \eqref{zerokappa} within the quasi-static regime. Then, the gradient remains uniformly bounded, independent of $r_1, r_2, \epsilon$, and $k$.
\end{corollary}

\subsection{Mathematical motivation and literature review}
The motivation for researching the gradient estimate problem for the Helmholtz equation stems from the theory of composite materials in photonics and phononics, where high-contrast building blocks are employed to manipulate and control waves in ways that are unattainable in conventional materials \cite{Lu2009,Li2004}.
Under an elaborate configuration (typically periodic, but not necessarily so), these composite materials exhibit exotic material properties, such as negative or high-contrast refractive indices on the macroscopic scale \cite{Ammari2019,Ammari2017,Ammari2017a}. 
Achieving these properties generally requires intricate assumptions regarding the volume fraction and orientation of the small inclusions.
Such composite materials are known as metamaterials or novel materials and have potential applications in wireless communications \cite{Tian2022,Xu2021}, super-resolution biomedical imaging \cite{Errico2015,Chen2017, BLLW9091}, quantum computing \cite{Siomau2012}, and invisibility cloaking \cite{ACK6055, L0069, LLZ2555, LL0165}. 
Given that these inclusions are typically nearly touching and exhibit sharp contrasts with the surrounding matrix, it is pertinent to question the reliability of these composite materials when interacting with electromagnetic or acoustic waves. 
From an engineering perspective, a critical physical quantity related to the stability of the composite material is stress, as a strong stress concentration can lead to the fracture of the material structure. 
In the context of transverse electromagnetic scattering, the degree of stress concentration is proportional to the physical quantity $\varepsilon^{-1}\nabla u$ (electromagnetic field intensity), rendering the gradient estimate problem for these metamaterials highly significant. 
Generally, the size of the building blocks is much smaller than the operating wavelength, facilitating the design of structures comprising numerous small inclusions \cite{AFGL007, LLZ0572, LZ0855}. This engineering setting prompts us to consider the gradient problem at the subwavelength scale or within the quasi-static regime, which is the focus of this article.

The gradient estimate problem in mathematics has been studied for over twenty years, originating from the stability analysis of classical fiber-reinforced composites \cite{Goree1967,Budiansky1984}. These studies have focused on a broad class of second-order elliptic equations, particularly the Laplace equation. Here, we highlight several results for the static case ($\omega = 0$ in \eqref{eq:helmhotz}) in various contexts related to composite materials, which are relevant to our current study. When material parameters are bounded away from zero and infinity, it has been demonstrated that the gradient field remains uniformly bounded even when inclusions touch or nearly touch \cite{Bonnetier2000, li2000gradient,Li2003}. In contrast, when the material parameters degenerate to extreme values, the gradient field generally becomes unbounded as the distance between adjacent inclusions, $\epsilon$, approaches zero. 
For perfect insulators ($\tilde{\varepsilon}$ approaches infinity), the optimal gradient blowup rate is of order $\epsilon^{-1/2}$ in two dimensions \cite{Ammari2007,yun2007estimates}, and the blowup rate decreases to $\epsilon^{-1/2+\beta}$ in dimensions greater than two \cite{Dong2024,Dong2024a}. The positive constant $\beta$ converges to $1/2$ as the dimension goes to infinity.
For perfect conductors ($\tilde{\varepsilon}$ approaches zero), it was shown that the optimal gradient blowup rate is of order $\epsilon^{-1/2}$ in two dimensions, order $|\epsilon \ln \epsilon|^{-1}$ in three dimensions, and order $\epsilon^{-1}$ in dimensions greater than three \cite{ yun2009optimal,bao2009gradient,bao2010gradient}.
We emphasize that in all of the aforementioned literature, the maximum principle generally plays a crucial role from both mathematical and physical perspectives. 
The maximum principle ensures that the gradient blowup phenomenon occurs along the shortest line between the adjacent inclusions, as the maximum and minimum must be taken on the boundaries of these inclusions. However, when we consider the gradient estimate problem for the wave field in the finite-frequency regime, the wave field may exhibit wave properties such as resonance and diffraction. Several studies address the gradient estimate problem for wave fields represented by solutions to the Helmholtz equation, where the maximum principle fails. The works in \cite{Ammari2020, LX3822, DLX7438, hu2025generating} investigate the gradient problem for the corresponding eigenfunctions under specific resonant conditions. In cases excluding resonance, other studies focus on gradient estimate problems in high-contrast materials and subwavelength regimes, as discussed in \cite{Deng2022a, Deng2022}.

In this paper, we derive the sharpest blowup condition and the optimal gradient estimate for the Helmholtz equation \eqref{mainequation} under each of the boundary conditions \eqref{infinitekappa}-\eqref{zerokappa} in the quasi-static regime. 
On the one hand, our results recover known static results and clarify the convergence relation as the frequency tends to zero. As a byproduct, we obtain gradient boundedness for the perfect electric conductors in the quasi-static regime.
Moreover, we demonstrate the influence of frequency on the gradient blowup, showing that wave frequency may mitigate the gradient blowup.
Although the Helmholtz equation does not model light-matter interactions with composite materials in three dimensions, it is insightful to study the three-dimensional case, which models acoustic wave scattering. We shall consider the gradient estimates in dimensions greater than two in a forthcoming paper.
In the subsequent sections, we present the proofs for the lower and upper bounds, respectively.

\section{The lower bound}
The key point for the lower bound is to estimate the potential difference on the boundaries of the two perfect conductors. We shall construct a quasi-static singular function  to characterize the potential difference.

\subsection{Inversion in a circle}
This subsection is devoted to some estimates two fixed points $\bp_1,\bp_2$ inside the disks $D_1,D_2$. The fixed points are defined via the inversion relation in a circle.

In a plane, the inverse of a point $\bx$ inside with respect to a reference circle $D$ with center $\bc$ and radius $r$ is a point $\bx'$ outside lying on the ray from $\bc$ through $\bx$. The  reflection relation is given by
\begin{equation*}
|\bc\bx| \cdot |\bc\bx'| = r^2.
\end{equation*}
This is called the circle inversion. If $\bx \in \partial D$, then $\bx'=\bx$. To extend inversion to  a global mapping that is also defined for the center $\bc$, it is necessary to introduce a point at infinity, a single point placed on all lines. This extension  interchanges the center $\bc$ and this point at infinity by definition.

The reflection relation for $\bx \in D_j$ is :
\begin{equation}\label{eq:reflect}
R_j(\bx) = \mathbf{c}_j + \frac{r_j^2}{|\bx - \mathbf{c}_j|}\frac{\bx - \mathbf{c}_j}{|\bx - \mathbf{c}_j|} \coma j=1,2.
\end{equation}
Straight inspection shows that there exist two fixed points $\bp_1,\bp_2$  satisfying $R_1(\bp_1) = \bp_2$ and $R_2(\bp_2)=\bp_1$. We denote them by
$$\bp_1 = (0,p_1) \coma \bp_2 = (0,p_2).$$ 
We now derive some estimates for the two fixed points.
\begin{lemma} \label{lem:fixp}
Let $j=1,2$. If $\min \{r_1, r_2\} \gg \epsilon$, then
\begin{align*}
p_j = (-1)^{j+1}2 \sqrt{\frac{r_1r_2}{r_1+r_2}} \sqrt{\epsilon} + \frac{r_1-r_2}{r_1+r_2}\epsilon +\ocal \left( \frac{\epsilon^{3/2}}{\min \{r_1^{1/2},r_2^{1/2}\}}  \right).
\end{align*}
If $\min \{r_1,r_2\} = \ocal (\epsilon)$, then
\begin{equation*}
 p_1 \sim \epsilon \coma p_2 \sim -\epsilon,
\end{equation*}
where the minus sign indicates that $p_2$ is negative.
\end{lemma}
\begin{proof}
The reflection relation \eqref{eq:reflect} and a straight computation show that  $p_1,p_2$ satisfy 
\begin{equation*}
(r_1+r_2+2\epsilon) x^2 + 2\epsilon(r_2-r_1)x -\epsilon(4r_1r_2+3r_1\epsilon+3r_2\epsilon+2\epsilon^2) = 0.
\end{equation*}
Solving this quadratic equation yields
\begin{equation}\label{eq:pp}
p_1 =  \sqrt{\epsilon} \; \frac{2d+\sqrt{\epsilon}(r_1-r_2)}{r_1+r_2+2\epsilon} > 0 \coma
p_2 =  -\sqrt{\epsilon} \; \frac{2d - \sqrt{\epsilon}(r_1-r_2)}{r_1+r_2+2\epsilon} < 0,
\end{equation}
where
\begin{equation*}
d = \sqrt{r_1r_2(r_1+r_2) +(r_1^2+r_2^2+3r_1r_2)\epsilon + 2(r_1+r_2)\epsilon^2 +\epsilon^3}.
\end{equation*}

When $\min \{r_1,r_2\} \gg \epsilon$, we have
\begin{align*}
d &= \sqrt{r_1r_2(r_1+r_2)} \left(1 + \frac{r_1^2+r_2^2+3r_1r_2}{2r_1r_2(r_1+r_2)}\epsilon +  \ocal \left(\frac{r_1^2+r_2^2}{r_1^2r_2^2}\epsilon^2 \right) \right)\\
&= \sqrt{r_1r_2(r_1+r_2)} \left(1 +  \ocal \left(\frac{\epsilon}{\min\{r_1,r_2\}}\right) \right).
\end{align*}
Substituting this estimate into equation \eqref{eq:pp} gives
\begin{align*}
p_1 &= 2\sqrt{\frac{r_1r_2}{r_1+r_2}}\sqrt{\epsilon}\left(1 + \frac{r_1-r_2}{2\sqrt{r_1r_2(r_1+r_2)}}\sqrt{\epsilon} +\ocal \left(   \frac{\epsilon}{\min \{r_1, r_2\} } \right)  \right) \\
&=2 \sqrt{\frac{r_1r_2}{r_1+r_2}} \sqrt{\epsilon} + \frac{r_1-r_2}{r_1+r_2}\epsilon +\ocal \left( \frac{\epsilon^{3/2}}{\min \{r_1^{1/2},r_2^{1/2}\}}  \right).
\end{align*}
The corresponding estimate for $p_2$ follows in the same way from the expansion of $d$.

When $\min \{r_1,r_2\} = \ocal (\epsilon)$, we have
\begin{equation*}
d \sim \sqrt{\epsilon} \cdot \max \{r_1,r_2,\epsilon\}.
\end{equation*}
Consequently, 
\begin{equation*}
\frac{2d \pm \sqrt{\epsilon}(r_1-r_2)}{r_1+r_2+2\epsilon} \sim \frac{ 2\max \{r_1,r_2,\epsilon\} \pm (r_1-r_2) }{r_1+r_2+2\epsilon}\sqrt{\epsilon} \sim \sqrt{\epsilon},
\end{equation*}
and hence we obtain the desired estimates of $p_1,p_2$. The proof is finished.
\end{proof}

We next derive some quantitative estimates for the distances between the fixed points and the centers of the disks.
\begin{lemma}\label{lem:distpc} 
Let $j=1,2$. 
If $\min \{r_1,r_2\} \gg \epsilon$, then
\begin{align*}
\abs{\bp_j -\bc_1}=& r_1 +  (-1)^j 2 \sqrt{\frac{r_1r_2}{r_1+r_2}} \sqrt{\epsilon} + 2\frac{r_2}{r_1+r_2}\epsilon+\ocal \left( \frac{ \epsilon^{3/2}}{\min \{r_1^{1/2},r_2^{1/2}\}}  \right),  \\
\abs{\bp_j -\bc_2} =& r_2 - (-1)^j 2 \sqrt{\frac{r_1r_2}{r_1+r_2}} \sqrt{\epsilon} + 2\frac{r_1}{r_1+r_2}\epsilon +\ocal \left( \frac{ \epsilon^{3/2}}{\min \{r_1^{1/2},r_2^{1/2}\}}  \right) .
\end{align*}
If $\min \{r_1,r_2\} = \ocal (\epsilon)$, then
\begin{align*}
&\abs{\bp_{3-j} -\bc_j} - r_j \sim  \epsilon \coma \abs{\bp_j -\bc_j}  \sim \min\left\{\frac{r_j}{\epsilon},1\right\}r_j.
\end{align*}
\end{lemma}
\begin{proof}
We first consider the case $\min \{r_1,r_2\} \gg \epsilon$.   From the geometry it holds that
\begin{equation}\label{eq:pc}
\abs{\bp_j -\bc_1} = r_1 +\epsilon - p_j \coma \abs{\bp_j -\bc_2} = p_j +r_2 +\epsilon.
\end{equation} 
Substituting the estimates for $p_j$ into Lemma \ref{lem:fixp}  yields the estimates.

When $\min \{r_1,r_2\} = \ocal (\epsilon)$, the formulas \eqref{eq:pc} give
\begin{equation*}
\abs{\bp_{3-j} - \bc_j} = r_j+\epsilon+(-1)^j p_{3-j}.
\end{equation*}
Together with   \ref{lem:fixp}, it implies
\begin{equation*}
\abs{\bp_1 - \bc_2} - r_2 \sim \epsilon \coma \abs{\bp_2 - \bc_1} - r_1 \sim \epsilon.
\end{equation*}
For convenience, we set
\begin{equation*}
r_1^*  =\max\{r_1,r_2\}\coma r_2^*  =\min\{r_1,r_2\}.
\end{equation*}
Then
\begin{equation*}
\abs{\bp_j - \bc_j} =r_j+\epsilon + (-1)^jp_j= \frac{r_j^2+r_1^*r_2^* +2\epsilon(r_1^*+r_2^*) +2\epsilon^2-2\sqrt{\epsilon}d}{r_1^*+r_2^*+2\epsilon}.
\end{equation*}
We next distinguish three subcases: $r_1^* \gg \epsilon ,r_1^* \sim \epsilon, r_1^* \ll \epsilon$.

When $r_1^* \gg \epsilon$,  we have $r_2^* =\ocal (\epsilon)$ and 
\begin{equation*}
\begin{aligned}
\displaystyle  d &= \sqrt{(r_1^*)^2(r_2^*+\epsilon) +r_1^*((r_2^*)^2+2\epsilon^2+3r_2^*\epsilon) + \epsilon((r_2^*)^2+2r_2^*\epsilon+\epsilon^2)} \medskip\\
\displaystyle &= r_1^*\sqrt{r_2^*+\epsilon} + \frac{r_2^*+2\epsilon}{2}\sqrt{r_2^*+\epsilon}  - \frac{(r_2^*)^2}{8r_1^*}\sqrt{r_2^*+\epsilon}+ \ocal \left(\frac{(2\epsilon+r_2^*)(r_2^*)^2}{(r_1^*)^2}\sqrt{r_2^*+\epsilon} \right). \medskip
\end{aligned}
\end{equation*}
Hence,
\begin{align*}
\displaystyle  r_j^2+r_1^*r_2^* +2\epsilon(r_1^*+r_2^*) +2\epsilon^2-2\sqrt{\epsilon}d  = r_j^2 +  r_1^* (r_2^*+2\epsilon-2\sqrt{\epsilon(r_2^*+\epsilon)})  \medskip\\
\displaystyle   + 2\epsilon^2 +2\epsilon r_2^* - (r_2^*+2\epsilon)\sqrt{\epsilon(r_2^*+\epsilon)} + \ocal \left( \frac{(r_2^*)^2\sqrt{\epsilon(r_2^*+\epsilon)}}{r_1^*}\right). \medskip 
\end{align*}
A second-order estimate yields
\begin{equation*}
 r_1^* (r_2^*+2\epsilon-2\sqrt{\epsilon(r_2^*+\epsilon)}) =  r_1^* \left(\frac{(r_2^*)^2}{4\epsilon}  + \ocal \left( \frac{(r_2^*)^3}{\epsilon^2}\right)\right).
\end{equation*}
Thus, when  $r_1^* \gg \epsilon$,
\begin{equation*}
\abs{\bp_j - \bc_j}  \sim\frac{r_j^2}{r_1^*} + \frac{(r_2^*)^2}{4\epsilon} \sim \min \left\{\frac{r_j}{\epsilon}, 1 \right\} r_j.
\end{equation*}
When $r_1^* =  \ocal(\epsilon)$, we have
\begin{equation*}
\begin{aligned}
 \displaystyle d &= \sqrt{\epsilon^3+ 2(r_1^*+r_2^*)\epsilon^2+((r_1^*)^2+(r_2^*)^2+3r_1^*r_2^*)\epsilon +  r_1^*r_2^*(r_1^*+r_2^*)  } \medskip\\
\displaystyle &=\epsilon^{3/2} \left(1 + \frac{r_1^*+r_2^*}{\epsilon} + \frac{r_1^*r_2^*}{2\epsilon^2} - \ocal \left(\frac{(r_1^*)^2(r_2^*)^2}{8\epsilon^4} \right) \right). \medskip
\end{aligned}
\end{equation*}
That implies
\begin{equation*}
\abs{\bp_j - \bc_j}  \sim\frac{r_j^2}{\epsilon} +\frac{(r_1^*)^2(r_2^*)^2}{8\epsilon^3} \sim  \frac{r_j}{\epsilon}r_j.
\end{equation*}
Summarizing all the above subcases, we conclude that
\begin{equation*}
\abs{\bp_j -\bc_j}  \sim \min\left\{\frac{r_j}{\epsilon},1\right\}r_j.
\end{equation*}
\end{proof}

\subsection{Singular function}
The quasi-static singular function was first introduced in \cite{Deng2022a} for high-contrast transverse electromagnetic scattering. Here we  use it to estimate the potential difference in the extreme cases \eqref{infinitekappa} and \eqref{onekappa}. 
Let $\Gamma_k$ be the fundamental solution of two-dimensional Helmholtz equation defined by
\begin{equation*}
\Gamma_k (\bx) = - \frac{\mi}{4}H_0^{(1)}(k|\bx|) \coma \bx \in \bbr^2,
\end{equation*}
where $H_0^{(1)}$ is the Hankel function of the first kind of order zero. We define the quasi-static singular function by
\begin{equation}\label{sys:hk}
h_k(\bx)  =  \Gamma_k (\bx-\bp_1) - \Gamma_k (\bx-\bp_2) .
\end{equation}
It can be verified  that $h_k$ satisfies
\begin{equation*}
\begin{cases}
\displaystyle \Delta h_k +k^2 h_k = 0  \hspace{122pt} \textrm{in} \quad \bbr^2\backslash \overline{D_1 \cup D_2}, \medskip \\
\displaystyle \int_{\partial D_j} \frac{\partial h_k}{\partial \nu} \dx{\mathbf{s}} + k^2 \int_{D_j} h_k \dx{\bx}= (-1)^{j+1}   \quad \textrm{on} \quad \partial D_j, j=1,2, \medskip \\
\displaystyle h_k   \quad \text{\rm satisfies the Sommerfeld radiation condition}.\medskip
\end{cases}
\end{equation*}

When  $k|\bx| \ll 1$, $\Gamma_k(\bx)$  admits  the  following asymptotic expansion:
\begin{equation*}
\Gamma_k (\bx)  = \frac{1}{2\pi} \ln|\bx| + \frac{1}{2\pi}\ln \frac{k}{2} + \frac{\gamma}{2\pi} - \frac{\mi}{4} + \sum_{j=1}^\infty b_j\left( \ln (k|\bx|) +c_j \right) (k |\bx|)^{2j} ,
\end{equation*}
where
\begin{equation*}
b_j = \frac{(-1)^j}{2\pi}\frac{1}{2^{2j}(j!)^2} \coma c_j = \gamma -\ln 2-\frac{\pi}{2}\mi - \sum_{i=1}^j\frac{1}{i}.
\end{equation*}
In the quasi-static regime, we have $k |\bx-\bp_j| \ll 1,j=1,2$. Hence the quasi-static singular function has the following asymptotic expansion:
\begin{align*}
h_k(\bx) :=h_0(\bx) +h_1(\bx) = \frac{1}{2\pi}\ln \frac{|\bx-\bp_1|}{|\bx-\bp_2|}  + h_1(\bx),
\end{align*}
where
\begin{equation*}
 h_1(\bx)  =  \frac{1}{8\pi} \sum_{j=1}^2 (-1)^j\ln k\abs{\bx - \bp_j}(k|\bx - \bp_j|)^2 + \ocal   \left(k^2 \epsilon(\min\{r_1,r_2\} + \epsilon) \right).
\end{equation*}

Next we investigate several properties for the singular function with respect to the disks $D_1$ and $D_2$ in the quasi-static regime.
\begin{proposition}\label{pro:value}
In the quasi-static regime, $h_k(\bx)$ is equal to a constant plus a high-order term on each boundary $\partial D_j, j=1,2$,
\begin{equation*}
h_k(\bx) = C_j + \ocal\left(  k^2r_j^2\ln (kr_j) \tilde{h}_j \right) , \bx \in \partial D_j.
\end{equation*}
If $\min \{r_1,r_2\} \gg \epsilon$, then
\begin{equation*}
C_j = \frac{ (-1)^{j}}{\pi} \sqrt{\frac{r_{3-j}}{r_j(r_1+r_2)}}\sqrt{\epsilon} + \ocal \left( \frac{\epsilon^{3/2}}{r_j\min \{r_1^{1/2},r_2^{1/2}\}} \right) \coma \tilde{h}_j =\sqrt{\frac{r_{3-j}}{r_j(r_1+r_2)}}\sqrt{\epsilon}.
\end{equation*}
If $\min \{r_1,r_2\}= \ocal (\epsilon)$, then
\begin{equation*}
C_j \sim
\begin{cases}
\displaystyle (-1)^{j} \epsilon/r_j \coma &r_j \gg \epsilon, \medskip\\
\displaystyle (-1)^{j} \coma &r_j \sim \epsilon, \medskip\\
\displaystyle (-1)^{j} \ln \left(\epsilon/r_j \right) , &r_j \ll \epsilon.\medskip
\end{cases}
 \coma  \tilde{h}_j = \frac{\epsilon \max \{r_j,\epsilon\}}{r_j^2}.
\end{equation*}
\end{proposition}
\begin{proof}
By using the properties of the fixed points, we obtain
\begin{equation}\label{eq:cc}
\frac{|\bx-\bp_1|}{|\bx-\bp_2|} = 
\begin{cases}
\displaystyle \frac{r_1}{|\bp_2 - \bc_1|} \coma  \bx \in \partial D_1 , \medskip\\
\displaystyle \frac{|\bp_1 - \bc_2|}{r_2}  \coma  \bx \in \partial D_2  \medskip.
\end{cases}
\end{equation}

When $\min \{r_1,r_2\} \gg \epsilon$ and  $\bx \in  \partial D_1$, we have
\begin{equation*}
 \frac{|\bp_2 - \bc_1|}{r_1} =  1 + 2\sqrt{\frac{r_2\epsilon}{r_1(r_1+r_2)}} + 2 \frac{r_2\epsilon}{r_1(r_1+r_2)} + \ocal \left( \frac{\epsilon^{3/2}}{r_1\min\{r_1^{1/2},r_2^{1/2}\}} \right),
\end{equation*}
and hence
\begin{align*}
C_1  =  \frac{1}{2\pi}\ln \frac{r_1}{|\bp_2 - \bc_1|}= -\frac{1}{\pi} \sqrt{\frac{r_2}{r_1(r_1+r_2)}}\sqrt{\epsilon} + \ocal \left( \frac{\epsilon^{3/2} }{r_1\min\{r_1^{1/2},r_2^{1/2}\}}\right).
\end{align*}
When $\bx  \in \partial D_2$, we obtain
\begin{align*}
C_2 = \frac{1}{2\pi} \ln \frac{|\bp_1 - \bc_2|}{r_2}= \frac{1}{\pi} \sqrt{\frac{r_1}{r_2(r_1+r_2)}}\sqrt{\epsilon} + \ocal \left( \frac{1}{r_2\min\{r_1^{1/2},r_2^{1/2}\}}\epsilon^{3/2} \right).
\end{align*}
For the high-order term, when  $\bx \in \partial D_1$, it follows from equation \eqref{eq:cc} that
\begin{equation*}
\ln k\abs{\bx - \bp_2}(k|\bx - \bp_2|)^2 = \frac{\abs{\bp_2 - \bc_1}^2}{r_1^2}  (k\abs{\bx-\bp_1})^2 \left(\ln  \frac{|\bp_2 - \bc_1|}{r_1} +\ln k\abs{\bx-\bp_1}\right).
\end{equation*}
Hence 
\begin{align*}
h_1(\bx) =& \left(\frac{\abs{\bp_2 - \bc_1}^2}{r_1^2} - 1 \right)  \ln k\abs{\bx-\bp_1}(k\abs{\bx-\bp_1})^2 \\
&+ \frac{\abs{\bp_2 - \bc_1}^2}{r_1^2}\ln  \frac{|\bp_2 - \bc_1|}{r_1}(k\abs{\bx-\bp_1})^2  + \ocal   \left(k^2\epsilon\min\{r_1,r_2\}  \right).
\end{align*}
From Lemma \ref{lem:distpc}, we first obtain
\begin{align*}
\frac{\abs{\bp_2 - \bc_1}^2}{r_1^2} - 1 =  4\sqrt{\frac{r_2\epsilon}{r_1(r_1+r_2)}}+ \ocal \left( \frac{r_2\epsilon}{r_1(r_1+r_2)}\right), \\
\frac{\abs{\bp_2 - \bc_1}^2}{r_1^2}\ln  \frac{|\bp_2 - \bc_1|}{r_1} = 2 \sqrt{\frac{r_2\epsilon}{r_1(r_1+r_2)}}+ \ocal \left( \frac{r_2\epsilon}{r_1(r_1+r_2)}\right) .
\end{align*}
We observe that $h_1(\bx)$ is negative and attains its minimum at $\bx = (0, 2r_1+\epsilon)$, so 
 \begin{align*}
&\max_{\bx \in \partial D_1}k^2\abs{\bx-\bp_1}^2 = k^2(2r_1+\epsilon-p_1)^2 \\
=& 4k^2r_1^2 \left( 1 - 2\sqrt{\frac{r_2\epsilon}{r_1(r_1+r_2)}} + \frac{3r_2\epsilon}{r_1(r_1+r_2)} + \ocal \left(\frac{1}{r_1\min\{r_1^{1/2},r_2^{1/2}\}}\epsilon^{3/2} \right)\right).
 \end{align*}
Hence
\begin{equation*}
\max_{\bx \in \partial D_1}\abs{h_1(\bx)} = -h_1(\bx  =(0,2r_1+\epsilon))   \sim 4\sqrt{\frac{r_2\epsilon}{r_1(r_1+r_2)}} k^2r_1^2\ln (kr_1)   \sim  \sqrt{\frac{\epsilon}{r_1}} k^2r_1^2\ln (kr_1).
\end{equation*}
Similarly, when $\bx \in \partial D_2$ we have 
\begin{align*}
h_1(\bx) =& \left(1-\frac{\abs{\bp_1 - \bc_2}^2}{r_2^2} \right)  \ln k\abs{\bx-\bp_2}(k\abs{\bx-\bp_2})^2 \\
&- \frac{\abs{\bp_1 - \bc_2}^2}{r_2^2}\ln  \frac{|\bp_1 - \bc_2|}{r_2}(k\abs{\bx-\bp_2})^2  + \ocal   \left(k^2 \epsilon\min\{r_1,r_2\} \right).
\end{align*}
and by  a straight calculation we obtain
\begin{equation*}
\max_{\bx \in \partial D_2}\abs{h_1(\bx)} = h_1(\bx  =(0,-2r_2-\epsilon)) \sim 4\sqrt{\frac{r_1\epsilon}{r_2(r_1+r_2)}} k^2r_2^2\ln (kr_2)   \sim  \sqrt{\frac{\epsilon}{r_2}} k^2r_2^2\ln (kr_2).
\end{equation*}

When $\min \{r_1,r_2\} = \ocal (\epsilon)$, we have
\begin{equation*}
C_j  = (-1)^{j}\frac{1}{2\pi}\ln \frac{r_j +\epsilon + |p_j|}{r_j} \sim  (-1)^{j} \ln \left( 1+  \frac{\epsilon}{r_j}\right).
\end{equation*}
For the high-order term, when $r_j \ll \epsilon$, we have
\begin{equation*}
\max_{\bx \in \partial D_j}\abs{h_1(\bx)} = \abs{h_1(\bx  =(0,2r_j+\epsilon))}  \sim \abs{\frac{\epsilon^2}{r_j^2}\ln (kr_j) k^2r_j^2} =\ocal \left( \ln (kr_j)\right),
\end{equation*}
where we have assumed that $\ln r_j = \ocal(\ln(k r_j))$. Otherwise, 
\begin{equation*}
\max_{\bx \in \partial D_j}\abs{h_1(\bx)}= \abs{h_1(\bx  =(0,2r_j+\epsilon))}  \sim \abs{\frac{\epsilon}{r_j} \ln (kr_j) k^2r_j^2} =\ocal \left( \frac{\epsilon}{r_j}\right).
\end{equation*}
\end{proof}

Next we estimate the integrals of the singular function over the two disks.
To this end, we introduce two auxiliary functions:
\begin{align*}
f_r (\bx) :=& \frac{1}{2\pi}\int_{D_r(\bc)} \ln |\bx -\by| \dx{\by} \coma \bx  \in \bbr^2, \\
g_r (\bx) :=& \frac{1}{8\pi}\int_{D_r(\bc)} \ln |\bx -\by|  (|\bx -\by|)^2 \dx{\by} \coma \bx  \in \bbr^2,
\end{align*}
where $D_r(\bc)$ is the disk with center $\bc$ and radius $r$.
\begin{lemma}\label{lem:intln}
$f_r (\bx)$ satisfies
\begin{equation*}
f_r (\bx) =
\begin{cases}
\displaystyle \frac{1}{4}( a^2 -  r^2) + \frac{1}{2}r^2\ln r \coma \bx \in D_r(\bc), \medskip \\
\displaystyle \frac{1}{2}r^2 \ln a \coma \bx \notin D_r(\bc). \medskip 
\end{cases}
\end{equation*}
$g_r (\bx)$ satisfies
\begin{equation*}
g_r (\bx) =
\begin{cases}
\displaystyle \frac{1}{64} \left(a^4 -  r^4 \right) + \frac{r^2}{16} \left( 2a^2\ln r +a^2 +r^2\ln r \right) \coma \bx \in D_r(\bc) , \medskip	\\
\displaystyle \frac{1}{8}r^2a^2\ln a+ \frac{r^4}{16} \left( \ln r + 1 \right)\coma \bx \notin D_r(\bc). \medskip
\end{cases}
\end{equation*}
where $a := \abs{\bx -\bc}$.

\end{lemma}

\begin{proof}
Note that the integrand in $f_r$ is the fundamental solution of the two-dimensional Laplace equation. When $\bx \in D_a(\bc)$, we have
\begin{equation*}
\Delta f_r (\bx)  = 1.
\end{equation*}
Since $ f_r (\bx) $ depends only on the distance between $\bx$ and the center $\bc$, $f_r(\bx)$
satisfies
 \begin{equation*}
\frac{d^2 f_r}{d a^2} + \frac{1}{a}\frac{d f_r}{d a} = 1 \coma a= \abs{\bx -\bc}.
\end{equation*}
The general solution is
\begin{equation*}
f_r (\bx) = \frac{a^2}{4} + c_1 \ln a + c_2.
\end{equation*}
We now identify the constants $c_1$ and $c_2$. Since $f_r$ is bounded at the center $\bc$, we have $c_1 = 0$ and
\begin{equation*}
c_2 = f_r(0) = \frac{1}{2\pi}\int_{D_r(\bc)} \ln \abs{\bx} \dx{\bx} = \frac{1}{2}r^2\ln r - \frac{1}{4}r^2.
\end{equation*}
Hence
\begin{equation*}
f_r(a) = \frac{1}{4}( a^2 -  r^2) + \frac{1}{2}r^2\ln r.
\end{equation*}

When  $\bx \notin D_r(\bc)$,  we have $\Delta f_r (\bx)  = 0$ and similarly obtain that
\begin{equation*}
\frac{\partial^2 f_r}{\partial a^2} + \frac{1}{a}\frac{\partial f_r}{\partial a} = 0 \coma a= \abs{\bx -\bc}.
\end{equation*}
The general solution is 
\begin{equation*}
f_r(a) = c_1 \ln a +c_2.
\end{equation*}
Since $f_r$ is continuous across the disk, it holds that
\begin{equation*}
\lim_{a \rightarrow r_+} f_r(a) = c_1\ln r +c_2 =\lim_{a \rightarrow r_-} f_r(a) = \frac{1}{2}r^2\ln r.
\end{equation*}
Hence $c_1 =r^2/2 \coma c_2=0$ and
\begin{equation*}
f_r(a) = \frac{1}{2}r^2 \ln a.
\end{equation*}

From the above results, when $\bx \in D_r(\bc)$ we have
\begin{align*}
\Delta g_r(\bx) &= \frac{1}{2\pi} \int_{D_r} (1+ \ln \abs{\bx-\by}) \dx{\by} = \frac{r^2}{2} + \frac{1}{4}( a^2 -  r^2) + \frac{1}{2}r^2\ln r \\
&=  \frac{1}{4}( a^2 +  r^2) + \frac{1}{2}r^2\ln r.
\end{align*}
Since $g_r$ also only depends on the distance $a = \abs{\bx - \bc}$, it satisfies
\begin{equation*}
\frac{\partial^2 g_r}{\partial a^2} + \frac{1}{a}\frac{\partial g_r}{\partial a} =  \frac{1}{4}( a^2 +  r^2) + \frac{1}{2}r^2\ln r.
\end{equation*}
The general solution is 
\begin{equation*}
g_r(\bx) = \frac{1}{64}a^4 + \frac{1}{16}r^2a^2 + \frac{1}{8}r^2\ln ra^2 +c_1 \ln a+c_2.
\end{equation*}
Using the value of $g_r(\bx)$ at the center $\bx =\bc$,  we obtain
\begin{equation*}
c_1 = 0 \coma c_2 = \frac{1}{16}r^4\ln r -\frac{1}{64}r^4.
\end{equation*}
Similarly, when $\bx \notin D_r$, we have
\begin{align*}
\Delta g_r(\bx) &= \frac{1}{2\pi} \int_{D_r} (1+ \ln \abs{\bx-\by}) \dx{\by} = \frac{a^2}{2} +  \frac{1}{2}a^2 \ln r.
\end{align*}
and therefore
\begin{equation*}
\frac{\partial^2 g_r}{\partial a^2} + \frac{1}{a}\frac{\partial g_r}{\partial a} = \frac{a^2}{2} +  \frac{1}{2}a^2 \ln r .
\end{equation*}
By an analogous argument we obtain
\begin{equation*}
g_r(\bx) = \frac{1}{8}r^2a^2\ln a + \frac{r^4}{16} \left( \ln r + 1 \right) \coma \bx \notin D_r.
\end{equation*}
\end{proof}

\begin{proposition} \label{pro:inte}
In the quasi-static regime, if $\min \{r_1,r_2\} \gg \epsilon$, then
\begin{align*}
\int _{D_1} h_k(\bx) \dx{\bx}&= -2r_1 \sqrt{\frac{r_1r_2}{r_1+r_2}}\sqrt{\epsilon} + \ocal \left(\min \{r_1,r_2\}\epsilon + k^2r_1^3\ln kr_1 \sqrt{\min \{r_1,r_2\}\epsilon} \right), \\
\int _{D_2} h_k(\bx)\dx{\bx} &=  2r_2 \sqrt{\frac{r_1r_2}{r_1+r_2}}\sqrt{\epsilon} + \ocal \left(\min \{r_1,r_2\}\epsilon + k^2r_2^3\ln kr_2 \sqrt{\min \{r_1,r_2\}\epsilon} \right),
\end{align*}
where $h_k$ is given in \eqref{sys:hk}.
If $\min \{r_1,r_2\} = \ocal (\epsilon)$ and we assume $r_j \ln r_j = \ocal (\epsilon)$ if $r_j \ll \epsilon$, then
\begin{equation*}
\int _{D_1} h_k(\bx) \dx{\bx} \sim - r_1\max \{r_1,\epsilon\} \coma 
\int _{D_2} h_k(\bx) \dx{\bx} \sim r_2 \max \{r_2,\epsilon\}.
\end{equation*}
\end{proposition}
\begin{proof}
by using Lemma \ref{lem:intln}, we have
\begin{align*}
\int_{D_1} h_0(\bx)  \dx{\bx}&= \frac{1}{4}( \abs{\bp_1 - \bc_1}^2 -  r_1^2) + \frac{1}{2}r_1^2\ln r_1 - \frac{1}{2}r_1^2 \ln (\abs{\bp_2 - \bc_1} \\
&= \frac{1}{4}( \abs{\bp_1 - \bc_1}^2 -  r_1^2)  - \frac{1}{2}r_1^2 \ln \frac{\abs{\bp_2 - \bc_1}}{r_1}.
\end{align*}
and \begin{align*}
\int_{D_1} h_1(\bx)  \dx{\bx} =& \frac{k^2r_1^2}{8} \left(\abs{\bp_2-\bc_1}^2\ln k\abs{\bp_2-\bc_1} - \abs{\bp_1-\bc_1}^2\ln kr_1 \right) +\\
 & \frac{k^2}{64} \left( 5r_1^2 + \abs{\bp_1-\bc_1}^2)( r_1^2 - \abs{\bp_1-\bc_1}^2  \right)  +\ocal \left(k^2r_1^2 \epsilon(\min\{r_1,r_2\} + \epsilon) \right). \\
\end{align*}

When $\min \{r_1,r_2\} \gg \epsilon$, straight calculations together with Lemma \ref{lem:fixp} yields
\begin{equation*}
\frac{1}{4}( \abs{\bp_1 - \bc_1}^2 -  r_1^2)  = -r_1 \sqrt{\frac{r_1r_2}{r_1+r_2}}\sqrt{\epsilon} + 2\frac{r_1r_2}{r_1+r_2}\epsilon +\ocal \left(\frac{r_1 }{\min \{r_1^{1/2},r_2^{1/2}\}}\epsilon^{3/2} \right) ,
\end{equation*}
and
\begin{equation*}
\frac{1}{2}r_1^2 \ln \frac{\abs{\bp_2 - \bc_1}}{r_1} = r_1 \sqrt{\frac{r_1r_2}{r_1+r_2}}\sqrt{\epsilon}+\ocal \left(\frac{r_1}{\min \{r_1^{1/2},r_2^{1/2}\}}\epsilon^{3/2} \right).
\end{equation*}
We have the integral of  $h_0$ in $D_1$,  
\begin{equation*}
\int_{D_1} h_0(\bx)  \dx{\bx}= -2r_1 \sqrt{\frac{r_1r_2}{r_1+r_2}}\sqrt{\epsilon} + 2\frac{r_1r_2}{r_1+r_2}\epsilon +\ocal \left(\frac{r_1}{\min \{r_1^{1/2},r_2^{1/2}\}}\epsilon^{3/2} \right).
\end{equation*}
By using Lemma \ref{lem:fixp} again, we have
\begin{align*}
\int_{D_2} h_0(\bx)  \dx{\bx} &= \frac{1}{2}r_2^2\ln \frac{|\bp_1 -\bc_2|}{r_2} + \frac{1}{4} (r_2^2 - |\bp_2 -\bc_2|^2) \\
&=2r_2 \sqrt{\frac{r_1r_2}{r_1+r_2}}\sqrt{\epsilon} - 2\frac{r_1r_2}{r_1+r_2}\epsilon +\ocal \left(\frac{r_2}{\min \{r_1^{1/2},r_2^{1/2}\}}\epsilon^{3/2} \right).
\end{align*}

We now estimate the integrals of $h_1(\bx)$.  Straight calculations together with Lemma \ref{lem:fixp} and Lemma \ref{lem:distpc} yield
\begin{equation*}
\ln k\abs{\bp_2-\bc_1}  = \ln kr_1 + 2\sqrt{\frac{r_2\epsilon}{r_1(r_1+r_2)}} + \ocal \left(\frac{r_2\epsilon}{r_1(r_1+r_2)} \right),
\end{equation*}
and
\begin{equation*}
\begin{aligned}
\abs{\bp_2-\bc_1}^2\ln k\abs{\bp_2-\bc_1} - \abs{\bp_1-\bc_1}^2\ln kr_1 =&8 r_1 \ln kr_1 \sqrt{\frac{r_1r_2\epsilon}{r_1+r_2}} \\
&+ \ocal \left(\ln kr_1 \frac{r_1r_2\epsilon}{r_1+r_2}  + r_1\sqrt{\frac{r_1r_2\epsilon}{r_1+r_2}}\right).
\end{aligned}
\end{equation*}
Hence, we have
\begin{align*}
\int_{D_1} h_1(\bx)  \dx{\bx} =& k^2r_1^3 \ln kr_1 \sqrt{\frac{r_1r_2\epsilon}{r_1+r_2}}+ \ocal \left(k^2r_1^3\ln kr_1 \frac{r_2\epsilon}{r_1+r_2}  + k^2r_1^3\sqrt{\frac{r_1r_2\epsilon}{r_1+r_2}}\right)\\
&+\frac{3}{8}k^2r_1^3\sqrt{\frac{r_1r_2\epsilon}{r_1+r_2}} + \ocal \left(\frac{k^2r_1^3r_2\epsilon}{r_1+r_2} \right)+\ocal \left(k^2r_1^2 \epsilon(\min\{r_1,r_2\} + \epsilon) \right) \\
=&k^2r_1^3 \ln kr_1 \sqrt{\frac{r_1r_2\epsilon}{r_1+r_2}} + \ocal \left(k^2r_1^3\ln kr_1 \frac{r_2\epsilon}{r_1+r_2}  + k^2r_1^3\sqrt{\frac{r_1r_2\epsilon}{r_1+r_2}}\right).
\end{align*}
We similarly obtain the integral over $D_2$,
\begin{align*}
\int_{D_2} h_1(\bx)  \dx{\bx} = -k^2r_2^3 \ln kr_2 \sqrt{\frac{r_1r_2\epsilon}{r_1+r_2}} + \ocal \left(k^2r_2^3\ln kr_2 \frac{r_1\epsilon}{r_1+r_2}  + k^2r_2^3\sqrt{\frac{r_1r_2\epsilon}{r_1+r_2}}\right).
\end{align*}
Combing the results for $h_0$ and $h_1$, we conclude that
\begin{align*}
\int _{D_1} h_k(\bx) \dx{\bx} 
=&  -2r_1 \sqrt{\frac{r_1r_2}{r_1+r_2}}\sqrt{\epsilon} + \ocal \left(\min \{r_1,r_2\}\epsilon + k^2r_1^3\ln kr_1 \sqrt{\min \{r_1,r_2\}\epsilon} \right) 
\end{align*}
and 
\begin{equation*}
\int _{D_2} h_k(\bx) \dx{\bx} = 2r_2 \sqrt{\frac{r_1r_2}{r_1+r_2}}\sqrt{\epsilon} + \ocal \left(\min \{r_1,r_2\}\epsilon + k^2r_2^3\ln kr_2 \sqrt{\min \{r_1,r_2\}\epsilon} \right).
\end{equation*}

Next, we consider the case $\min \{r_1,r_2\} = \ocal (\epsilon)$. If $r_j \ll \epsilon$, then
\begin{equation*}
\int_{D_j} h_0(\bx) \dx{\bx} \sim (-1)^j \left( \frac{1}{4} \left( 1 - \frac{r_j^2}{\epsilon^2}  \right)r_j^2 + \frac{1}{2}r_j^2\ln \frac{\epsilon}{r_j} \right) \sim (-1)^j  r_j\epsilon.
\end{equation*}
where we have used $r_j \ln r_j = \ocal (\epsilon) $ if $r_j \ll \epsilon$.
If $\epsilon = \ocal  (r_j)$, then
\begin{equation*}
\int_{D_j} h_0(\bx) \dx{\bx} \sim (-1)^j \left(\frac{1}{4}r_j^2 + \frac{1}{2}r_j^2\cdot \ln \left(1+\frac{\epsilon}{r_j}\right)\right) \sim  (-1)^{j}r_j^2.
\end{equation*}
 If $r_j \ll \epsilon$, the high-order term satisfies
\begin{align*}
\int_{D_j} h_1(\bx) \dx{\bx} &\sim k^2r_1^2  \left( \epsilon^2\ln k(r_1+\epsilon) - \frac{r_j^4}{\epsilon^2} \ln kr_j +  r_j^2 +\ocal \left( \epsilon(\min\{r_1,r_2\} + \epsilon) \right)\right) \\
&\sim r_j^2\cdot \ocal (k^2\epsilon^2\ln k\epsilon) \ll r_j\epsilon.
\end{align*}
If $\epsilon = \ocal  (r_j)$, then
\begin{equation*}
\int_{D_j} h_1(\bx) \dx{\bx} \sim k^2r_1^2 \left(r_j\epsilon\ln kr_j +r_j^2+\ocal \left( \epsilon(\min\{r_1,r_2\} + \epsilon) \right)\right) \ll r_j^2.
\end{equation*}
Since the high-order terms are of smaller order than the leading-order terms, we complete  the proof.
\end{proof}

\subsection{Proof of the lower bound}
The following lemma is critical for estimating the difference between $\lambda_1$ and $\lambda_2$.
\begin{lemma}\label{lem:poten}
Let $h_k$ be defined in  \eqref{sys:hk} and $u,u^i$ be the total field and incident field for the system \eqref{mainequation}, respectively, then 
\begin{equation*}
\int_{\partial D_1 \cup \partial D_2}\frac{\partial h_k}{\partial \nu} u - \frac{\partial u}{\partial \nu} h_k \dx{s} = u^i(\bp_2) - u^i(\bp_1).
\end{equation*}
\end{lemma}
\begin{proof}
We first  show that
\begin{equation}\label{eq:radia}
\lim_{r\rightarrow \infty}  \int_{\partial D_r} \frac{\partial h_k}{\partial \nu}u^s - \frac{\partial u^s}{\partial \nu} h_k \dx{\bx} = 0,
\end{equation}
where $u^s:=u-u^i$ is the radiating solution outside $D$. By the Sommerfeld radiation condition, it holds that
\begin{equation*}
\lim_{r\rightarrow \infty}  \int_{\partial D_r} \left|\frac{\partial u^s}{\partial \nu} -\mi ku^s \right|^2 \dx{s} = \lim_{r\rightarrow \infty}  \int_{\partial D_r} \left|\frac{\partial u^s}{\partial \nu} \right|^2 + k^2 \left| u^s \right|^2 +2k \Im \left(u^s \frac{\partial \overline{u^s}}{\partial \nu} \right) \dx{s}  = 0.
\end{equation*}
Let $D_r$ be a sufficiently large disk containing $D$ and  Green's formula in $D_r \backslash \overline{D}$ yields
\begin{equation*}
\int_{\partial D_r} u^s \frac{\partial \overline{u^s}}{\partial \nu} \dx{s} - \int_{\partial D} u^s \frac{\partial \overline{u^s}}{\partial \nu} \dx{s} = \int_{D_r \backslash \overline{D}} |\nabla u^s|^2 -k^2 (u^s)^2 \dx{\bx}.
\end{equation*}
Taking the imaginary part gives
\begin{equation*}
\lim_{r\rightarrow \infty}  \int_{\partial D_r} \left|\frac{\partial u^s}{\partial \nu} \right|^2 + k^2 \left| u^s \right|^2 \dx{s} = -2k   \int_{\partial D}  \Im \left(u^s \frac{\partial \overline{u^s}}{\partial \nu} \right) \dx{s}.
\end{equation*}
Both terms on the left-hand side are nonnegative and they must be individually bounded as $r \rightarrow \infty$.
Note that $h_k$ and $u^s$ are both radiating solutions. From the Sommerfeld radiation condition and Cauchy-Schwarz inequality, we can get 
\begin{equation*}
\lim_{r\rightarrow \infty}  \int_{\partial D_r} (u-u^i)\left(\frac{\partial h_k}{\partial \nu} - \mi k h_k\right) \dx{s} = 0.
\end{equation*}
Similarly,
\begin{equation*}
\lim_{r\rightarrow \infty}  \int_{\partial D_r} h_k \left(\frac{\partial (u-u^i)}{\partial \nu} - \mi k (u-u^i)\right) \dx{s} = 0.
\end{equation*}
Combining the above equalities we can prove equation  \eqref{eq:radia}.

By using equation \eqref{eq:radia} and Green's formula  again outside $D_1 \cup D_2$, we obtain
\begin{equation*}
\int_{\partial D_1 \cup \partial D_2} \frac{\partial h_k}{\partial \nu}(u-u^i) - \frac{\partial (u - u^i)}{\partial \nu} h_k \dx{\bx} =  0.
\end{equation*}
Extracting the terms related to $u^i$ and  Green's formula again inside $D_1 \cup D_2$ yields
\begin{align*}
\int_{\partial D_1 \cup \partial D_2} \frac{\partial  u^i}{\partial \nu} h_k - \frac{\partial h_k}{\partial \nu}u^i \dx{\bx} &= \int_{D_1 \cup D_2} (\Delta u^i +k^2 u^i)h_k - (\Delta h_k +k^2 h_k)u^i  \dx{\bx}\\
&=u^i(\bp_2) - u^i(\bp_1).
\end{align*}
This completes the proof.
\end{proof}
\begin{proof}[Proof of the lower bound]
When $\min \{r_1,r_2\} \gg \epsilon$,  we have $\abs{\bp_1-\bp_2} \sim \sqrt{\epsilon}$,and thus
\begin{equation}\label{est:ui}
\begin{aligned}
\displaystyle  u^i(\bp_2) - u^i(\bp_1) =& -\partial_{\bx_2} u^i(\mathbf{0}) (p_1-p_2) - \partial^{2}_{\bx_2} u^i(\mathbf{0}) \frac{p_2^2-p_1^2}{2} + \ocal \left(p_1^3+\abs{p_2}^3 \right)  \medskip \\
\displaystyle =&-4\sqrt{\frac{r_1r_2}{r_1+r_2}}\sqrt{\epsilon} \left(\partial_{\bx_2} u^i(\mathbf{0})+\ocal(\epsilon) \right). \medskip
\end{aligned}
\end{equation}
We next estimate $\abs{\lambda_1-\lambda_2}$ where $\lambda_j$ are determined by the boundary integral \eqref{onekappa}. From Proposition \ref{pro:value} and Proposition \ref{pro:inte} we obtain
\begin{equation}\label{est:value}
\begin{aligned}
\displaystyle  \int_{\partial D_j}\frac{\partial u}{\partial \nu} h_k \dx{s} &= \left(C_j +\ocal \left(k^2r_j^2\ln kr_j \sqrt{\frac{r_{3-j}}{r_j(r_1+r_2)}}\sqrt{\epsilon} \right) \right)\int_{\partial D_j}\frac{\partial u}{\partial \nu}\dx{s} \medskip\\
\displaystyle&= (-1)^{j+1} \lambda_j \ocal \left(k^2r_j^2\sqrt{\frac{\epsilon}{r_j}} \right), \medskip
\end{aligned}
\end{equation}
and
\begin{equation}\label{est:int}
\begin{aligned}
\displaystyle \int_{\partial D_j} \frac{\partial h_k}{\partial \nu} u \dx{s}=& \lambda_j \left((-1)^{j+1} - k^2 \int_{D_j} h_k \dx{\bx} \right) \medskip \\
\displaystyle =& (-1)^{j+1}\lambda_j \left(1 +\ocal \left(k^2r_j\sqrt{\min \{r_1,r_2\}\epsilon} \right)\right). \medskip
\end{aligned}
\end{equation}
Combining  \eqref{est:ui}, \eqref{est:int}, \eqref{est:value} and Lemma \ref{lem:poten}, we obtain
\begin{align*}
(\lambda_1 - \lambda_2) \left(1 + \ocal \left(k^2\left\{r_1^2,r_2^2\right\}\sqrt{\frac{\epsilon}{\min\{r_1,r_2\}}}   \right) \right) =-4\sqrt{\frac{r_1r_2\epsilon}{r_1+r_2}} \left(\partial_{\bx_2} u^i(\mathbf{0})+\ocal(\epsilon) \right),
\end{align*}
and hence
\begin{equation*}
\lambda_2 -\lambda_1 = 4\sqrt{\frac{r_1r_2}{r_1+r_2}}\sqrt{\epsilon} \Biggl(\partial_{\bx_2} u^i(\mathbf{0}) + \ocal \biggl(\epsilon + \partial_{\bx_2} u^i(\mathbf{0}) k^2\max \left\{r_1^2,r_2^2\right\}\sqrt{\frac{\epsilon}{\min\{r_1,r_2\}}} \biggr) \Biggr) .
\end{equation*}
The lower bound of the gradient blowup can be obtained by the mean value theorem.

When $\min \{r_1,r_2\} =\ocal \left(\epsilon \right)$, the relation $p_j \sim (-1)^{j+1}\epsilon$ implies that there exists a points $\bx_*$ lying between  $\bp_1$ and $\bp_2$ such that
\begin{equation*}
u^i(\bp_2) - u^i(\bp_1) \sim - \epsilon \cdot \partial_{\bx_2}u^i(\mathbf{\bx_*}) .
\end{equation*}
Further calculations yield
\begin{equation*}
\int_{\partial D_j} \frac{\partial h_k}{\partial \nu} u \dx{s}   = (-1)^{j+1}\lambda_j \left( 1+ \ocal( k^2r_j\max\{r_j,\epsilon\} \right),
\end{equation*}
and 
\begin{equation*}
\int_{\partial D_j}\frac{\partial u}{\partial \nu} h_k \dx{\bx} = (-1)^j\lambda_j \pi k^2r_j^2\left(C_j +\ocal (k^2\epsilon\max\{r_j,\epsilon\}\ln kr_j) \right).
\end{equation*}
When $\epsilon = \ocal (r_j) \coma  C_j = (-1)^{j+1}\ocal (1)$. otherwise if $\epsilon \gg r_j$, we have
\begin{equation*}
C_j \sim \ln \frac{\epsilon}{r_j} \sim -\ln r_j =\ocal (-\ln k r_j).
\end{equation*}
We have used $kr_j\ln kr_j = \ocal(k\epsilon)$ if $r_j \ll \epsilon$ and hence $\abs{k^2r_j^2 \ln k r_j} = \ocal(k^2 r_j\epsilon).$
Finally it holds that
\begin{equation*}
\lambda_2 -\lambda_1  =  \epsilon \cdot \partial_{\bx_2}u^i(\mathbf{\bx_*}) \; \ocal \left( 1+k^2\max\{r_1,r_2\}\max\{r_1,r_2, \epsilon\} \right).
\end{equation*}

When the boundary integral modifies to the boundary integral \eqref{infinitekappa}, the argument for the potential difference is the same except for the boundary integral of the normal derivative \eqref{est:value}. Since $ \partial_{\nu}u \in H^{-1/2}(\partial \Omega)$, the boundary integral  can be interpreted as  the dual pairing $\langle\cdot , \cdot\rangle_{H^{-1/2},H^{1/2}}$ on $\partial \Omega$. Hence it holds that
\begin{align*}
\displaystyle \int_{\partial D_j}\frac{\partial u}{\partial \nu} h_k \dx{s}   &= \ocal \left( k^4r_j^4\ln kr_j \sqrt{\frac{r_{3-j}}{r_j(r_1+r_2)}}\sqrt{\epsilon}\right) \coma \min \{r_1,r_2\} \gg \epsilon , \medskip\\
\displaystyle\int_{\partial D_j}\frac{\partial u}{\partial \nu} h_k \dx{s}   &= \ocal \left( k^4r_j^2 \epsilon \max\{r_j, \epsilon\} \ln kr_j \right) \coma \hspace{29pt}  \min \{r_1,r_2\} = \ocal (\epsilon).  \medskip
\end{align*}
By  combining \eqref{est:ui} and \eqref{est:int}, we obtain the same lower bound, which completes the proof.
\end{proof}

\section{The upper bound}
This section aims to demonstrate that the lower bound established in the previous section  is optimal up to  a bounded term. We outline the main ideas of the proof. We first truncate the exterior scattering system within a  sufficiently large bounded domain containing the two disks. The truncated system can be decomposed into two subsystems, as shown in \eqref{sys:u1} and \eqref{sys:u2}. The core idea is to construct an auxiliary function that captures the singular behavior of gradient  and to prove that the remaining part is uniformly bounded.


\subsection{Auxiliary construction}
Let  $R$ be sufficiently large and $D_R$ be a disk containing the disks $D_1$ and $D_2$, defined by:
$$D_R:=\{\bx \in \bbr^2  |\abs{\bx} \leqslant R\}.$$
The following  lemma justifies the validity of truncation.
\begin{lemma}
Let $u$ be the unique solution of the system \eqref{mainequation}, then $u$ satisfies
\begin{equation*}
\begin{cases}
\displaystyle \Delta u + k^2 u= 0    \quad &\mathrm{in}  \quad \Omega, \\
\displaystyle u = \lambda_j  &\mathrm{on} \quad  \partial D_j \coma j=1,2,\\
\displaystyle u  = f  &\mathrm{on} \quad  \partial D_R,
\end{cases}
\end{equation*}
where $\Omega := D_R\backslash\overline{D_1\cup D_2}$ and $ f\in C^\infty(\partial D_R)$.
\end{lemma}
\begin{proof}
In fact, $f : = \left. -u^i\right|_{\partial D_R} + f_1$  with $f_1$ denoting the boundary value defined on $\partial D_R $  of the radiating solution $u^s$.
Since $u^s$ satisfies the Sommerfeld radiation condition, Green's formula yields 
\begin{equation*}
u^s (\bx) = \int_{\partial D_1 \cup \partial D_2} \frac{\partial \Phi (\bx,\by)}{\partial \nu_\by} u^s(\by) - \Phi (\bx,\by) \frac{\partial u^s(\by)}{\partial \nu_\by}  \dx{s_\by},
\end{equation*}
where $\Phi (\bx,\by) = H_0^{(1)}(k\abs{\bx-\by})$ is the first Hankel function of order zero. 
From the asymptotic properties of Bessel function, we have
\begin{equation*}
D^n H_0^{(1)}(k\abs{\bx-\by})   = \ocal (\abs{\bx}^{-1/2})  \coma \mathrm{as} \quad \abs{\bx} \rightarrow \infty \coma\forall n \in \bbn.
\end{equation*}
A straight calculation also shows
\begin{equation*}
    \frac{\partial^n}{\partial\bx^n} \abs{\bx-\by} = \ocal (\abs{\bx}^{1-n})\coma \mathrm{as} \quad \abs{\bx} \rightarrow \infty \coma\forall n \in \bbn .
\end{equation*}
The incident field $u^i$ is analytic in the entire space, and consequently $ f\in C^\infty(\partial D_R)$. 
\end{proof}
Standard regularity estimates for the system \eqref{mainequation} (see \cite{Colton2019}) imply
\begin{equation*}
\|u\|_{C^1(\bbr^2\backslash \overline{D_R})} \leqslant \|u^i\|_{C^1(\bbr^2\backslash \overline{D_R})} + C\|f_1\|_{C(\partial D_R)}.
\end{equation*}
The regularity ensures the boundedness of $\nabla u$ outside $D_R$. Therefore, it  suffices to analyze $\nabla u$ inside $D_R$. 
To characterize the singular behavior of $u$ within $D_R$,
we decompose $u$ as $u =u_1 + u_2$, where  $u_1$ and $u_2$ are defined by the following boundary value problems:
\begin{equation}\label{sys:u1}
\begin{cases}
\displaystyle \Delta u_1 + k^2 u_1 = 0   & \mathrm{in}  \quad  \Omega , \\
\displaystyle  u_1 = \lambda_j &\mathrm{on} \quad  \partial D_j,j=1,2,\\
\displaystyle u_1 = 0  &\mathrm{on} \quad  \partial D_R,
\end{cases}
\end{equation}
and $u_2 := u - u_1$ satisfies
\begin{equation}\label{sys:u2}
\begin{cases}
\displaystyle \Delta u_2 + k^2 u_2 = 0   & \mathrm{in}  \quad  \Omega , \\
\displaystyle  u_2 = 0 &\mathrm{on} \quad  \partial D_j,j=1,2,\\
\displaystyle u_2 = f  &\mathrm{on} \quad  \partial D_R.
\end{cases}
\end{equation}
Without loss of generality, we choose $D_R$ such that $k^2$ is not a Dirichlet eigenvalue of $\Delta$ in $\Omega$.

We next provide a more detailed  geometric characterization of  the interval between the two disks. As shown in Figure \ref{fig:twodisks}, the sections of $\partial D_1$ and $\partial D_2$ along the interval can be characterized by 
\begin{equation*}
x_2 := (-1)^{j+1} \left(g_j(x_1 ) + \epsilon\right) =  (-1)^{j+1} \left(\epsilon + r_j - \sqrt{r_j^2 -x_1^2}\right) \coma \abs{x_1} < r_0, \coma j =1,2,
\end{equation*}
where we set $r_0 : =\min\{r_1,r_2\}/2$.
For a point $\bz = (z_1, z_2)$ lying inside the intervalbetween the disks, we define the neighborhood

\begin{equation*}
\Omega_{r(\bz)} := \{\bx =(x_1,x_2) \in \bbr^2 | -\epsilon - g_2(x_1) < x_2 < \epsilon +g_1(x_1) \coma \abs{x_1 -z_1}<r\}.
\end{equation*}
 For simplicity, we denote $\Omega_{r}$ if $\bz$ is the origin.
The vertical distance between $\partial D_1$ and $\partial D_2$ is given by
\begin{equation}\label{eq:delta}
\delta(x_1) := 2\epsilon +g_1(x_1) + g_2(x_1).
\end{equation}

We define an auxiliary function $p_1(\bx) $ in $\Omega_{r_0}$ by
\begin{equation}\label{eq:p1}
p_1(\bx) = \frac{x_2+\epsilon+g_2(x_1)}{\delta(x_1)} \quad \mathrm{in} \quad  \Omega_{r_0}.
\end{equation}
Clearly, $p_1(\bx) = 1$ on $\partial D_1 \coma  p_1(\bx) = 0$ on $\partial D_2$ for $\abs{x_1} <r_0$.
We next construct a function $\tu_1$ to capture the singular behavior of $u_1$ by using $p_1$:
\begin{equation}\label{eq:tu1}
\tu_1(\bx) = (\lambda_1 - \lambda_2)p_1(\bx) + \lambda_2 \quad \mathrm{in} \quad  \Omega_{r_0},
\end{equation}
It follows  that $\tu_1(\bx) = \lambda_j$ on $\partial D_j$ for $\abs{x_1} <r_0$. We can extend $\tu_1(\bx)$ to $\Omega \backslash \overline{\Omega_{r_0}}$ such that $\left.\tu_1\right|_{\partial D_j} = \lambda_j$ ,$\left.\tu_1\right|_{\partial D_R} = 0$ and the following bound holds:
\begin{equation*}
\|\tu_1(\bx)\|_{C^{2}(\Omega \backslash \overline{\Omega_{r_0}})} \leqslant C(|\lambda_1| +|\lambda_2|).
\end{equation*}

\begin{lemma}\label{lem:u1}
For any $\bz$-neighbourhood $\Omega_{s(\bz)} \subset \Omega_{r_0}$, we have
\begin{align*}
 &\|\tu_1\|_{L^2(\Omega_{s(\bz)})} \leqslant C \sqrt{\left(\delta(z_1) + \frac{s^2}{\min\{r_1,r_2\}} \right)s} (\abs{\lambda_1} + \abs{\lambda_2}), \\
& \| \Delta \tu_1\|_{L^2(\Omega_{s(\bz)})} \leqslant C \frac{1}{\min\{r_1,r_2\}} \sqrt{\frac{s}{\delta(z_1)}} \abs{\lambda_1-\lambda_2},
\end{align*}
where $C$ is a constant independent of $r_1,r_2,\epsilon$ and $s$.
\end{lemma}
\begin{proof}
From the equation \eqref{eq:delta}, we observe that
\begin{equation*}
\int_{\Omega_{s(\bz)}}  \tu_1^2 \dx{\bx} = \int_{\abs{x_1-z_1}\leqslant s} \dx{x_1}\int_{-\epsilon-g_2(x_1)}^{\epsilon+g_1(x_1)}\dx{x_2} \left( (\lambda_1 -\lambda_2) p_1(x_1,x_2) +   \lambda_2 \right)^2.
\end{equation*}
Integrating in the $x_2$ direction yields
\begin{equation*}
\int_{\Omega_{s(\bz)}}  \tu_1^2 \dx{\bx}\leqslant \frac{1}{3}(\lambda_1 +\lambda_2)^2\int_{\abs{x_1-z_1}\leqslant s}\delta(x_1) \dx{x_1},
\end{equation*}
where we have used
\begin{equation*}
 \int_{-\epsilon-g_2(x_1)}^{\epsilon+g_1(x_1)}x_2^2 \dx{x_2}  = \frac{1}{3}\delta^3(x_1) - \delta(x_1)(\epsilon +g_1(x_1))(\epsilon+g_2(x_1)) .
\end{equation*}
Since $\delta(x_1)$ is  convex and $g_j(x_1) < x_1^2/r_j$ when $\abs{x_1}\leqslant r_0$, we obtain
\begin{align*}
\int_{\Omega_{s(\bz)}}  \tu_1^2 \dx{\bx} &\leqslant \frac{1}{6}(\lambda_1 +\lambda_2)^2  (\delta(z_1-s) +\delta(z_1+s))s  \\
&\leqslant \frac{1}{6}(\lambda_1 +\lambda_2)^2 \left( 2\delta(z_1) +\frac{s^2}{r_1}+\frac{s^2}{r_2}\right)s.
\end{align*}

It is noted from  the equation (\ref{eq:p1}) that $\Delta \tu_1$ is independent of $x_2$.  A straight calculation gives
\begin{align*}
\Delta \tu_1 &= \frac{1}{\delta^2} \left(\delta \left(g''_2 - \frac{x_2+\epsilon+g_2}{\delta} \delta{''}\right) + 2\delta' \left(\frac{x_2+\epsilon+g_2}{\delta} \delta{'} - g'_2  \right)\right) .
\end{align*}
Observing the fact $x_2 \leqslant \epsilon +g_1(x_1)$, we  find 
\begin{equation*}
\Delta \tu_1 \leqslant \frac{2}{\delta^2} \left(\delta\delta{''} + 2{\delta'}^2 \right).
\end{equation*}
Using the convexity of $\delta(x_1)$ again we  deduce
\begin{align*}
\int_{\Omega_{s(\bz)}} (\Delta \tu_1)^2 \dx{\bx} &\leqslant 4\abs{\lambda_1-\lambda_2}^2 \int_{\Omega_{s(\bz)}} \frac{(\delta''\delta +2(\delta')^2)^2}{\delta^4} \dx{\bx} \\
&\leqslant  C \abs{\lambda_1-\lambda_2}^2 \frac{1}{\min\{r_1^2,r_2^2\}} \frac{s}{\delta (z_1)},
\end{align*}
where we have used  the inequality
\begin{equation*}
\delta''\delta > \frac{1}{2}x_1^2 \left(\frac{1}{r_1^2-x_1^2} + \frac{1}{r_2^2-x_1^2} + 2r_1r_2(r_1^2-x_1^2)^{-\frac{3}{2}}(r_2^2-x_1^2)^{-\frac{3}{2}}\right) > \frac{1}{2}(\delta')^2.
\end{equation*}
\end{proof}

\subsection{Local gradient estimate}
 In this subsection, we prove Lemma \ref{lemma:graw} to obtain a local $L^2$ estimate for the gradient in the interval between the two disks. To this end, we  establish two Poincar\'e-type inequalities. 
\begin{lemma}
Let $w \in H^1(\Omega)$  and $\Omega_{s(\bz)} \subset \Omega_{r_0}$.
If $w = 0$ on $\partial D_2$ for $\abs{x_1} \leqslant r_0$, then
\begin{equation}\label{eq:poincare1}
\|w\|_{L^2(\Omega_{s(\bz)})} \leqslant \max_{\bx \in \Omega_{s(\bz)}}\delta(x_1) \| \nabla w\|_{L^2(\Omega_{s(\bz)})}.
\end{equation}
If $w = 0$ on $\partial D_R$, then
\begin{equation}\label{eq:poincare2}
\|w\|_{L^2(\Omega)} \leqslant R/2 \| \nabla w\|_{L^2(\Omega)}.
\end{equation}
\end{lemma}
\begin{proof}
By the Newton-Leibniz formula, we first have
\begin{equation*}
w(x_1,x_2) = w(x_1,-\epsilon-g_2) + \int_{-\epsilon-g_2}^{x_2} \partial_y w(x_1,y) \dx{y}.
\end{equation*}
Since $w$ vanishes on $\partial D_2$, it follows that
\begin{equation*}
w^2(x_1,x_2) \leqslant (x_2+\epsilon+g_2)  \int_{-\epsilon-g_2}^{x_2} \left(\partial_y w(x_1,y)\right)^2 \dx{y} \leqslant \delta(x_1) \int_{-\epsilon-g_2}^{\epsilon+g_1}\left(\partial_y w(x_1,y)\right)^2  \dx{y}.
\end{equation*}
Integrating over $\Omega_{s(\bz)}$ yields
\begin{equation*}
\|w\|_{L^2(\Omega_{s(\bz)})}^2 \leqslant \max_{\bx \in \Omega_{s(\bz)}}\delta^2(x_1) \int_{\Omega_{s(\bz)}}  \abs{\partial_{x_2} w}^2\dx{\bx}.
\end{equation*}
This establishes the first Poincar\'e-type inequality.

For the second inequality, we again apply the Newton–Leibniz formula and the vanishing condition on $\partial D_R$ to obtain
\begin{equation*}
w^2(r,\theta) = \left(- \int_{r}^R \partial_y w(y,\theta) \dx{y}\right)^2 \leqslant (R-r) \int_{r}^R  \left(\partial_y w(y, \theta)\right)^2 \dx{y}.
\end{equation*}
Since $r \leqslant y$, we further have
\begin{equation*}
rw^2(r,\theta) \leqslant  (R-r) \int_{r}^Ry\left(\partial_y w(y, \theta)\right)^2 \dx{y} \leqslant  (R-r) \int_{0}^R y\left(\partial_y w(y, \theta)\right)^2 \dx{y}.
\end{equation*}
Integrating over $\Omega$ yields
\begin{equation*}
\|w\|_{L^2(\Omega_{s(\bz)})}^2 \leqslant \frac{1}{2}R^2 \int_{\Omega} \left(\partial_r w(r, \theta)\right)^2r \dx{r} \leqslant  \frac{1}{2}R^2 \int_{\Omega}  \abs{\nabla w}^2\dx{\bx}.
\end{equation*}
Thus the second Poincar\'e-type inequality holds. 
\end{proof}

Let $w_1=u_1 - \tu_1$, then  $w_1$ satisfies
\begin{equation}\label{eq:wj}
\begin{cases}
\Delta w_1 +k^2 w_1= -\Delta \tu_1 -k^2 \tu_1 \quad &\textrm{in} \quad  \Omega, \\
w_1= 0  &\textrm{on} \quad \partial \Omega.
\end{cases}
\end{equation}

\begin{lemma}
In the quasi-static regime, we have
\begin{equation*}
\|\nabla w_1\|_{L^2(\Omega)} \leqslant C (\abs{\lambda_1} +\abs{\lambda_2}).
\end{equation*}
where $C>0$ is some constant independent of $r_1,r_2,\epsilon$ and $k$.
\end{lemma}
\begin{proof}
Multiplying the equation (\ref{eq:wj})  by $w_1$ and using Green’s first identity in $\Omega$ yields
\begin{equation*}
\begin{aligned}
 \int_{\Omega} \abs{\nabla w_1}^2 \dx{\bx}=& \int_{\Omega} w_1 \Delta \tu_1 \dx{\bx} +k^2\int_{\Omega}w_1\tu_1 +w_1^2\dx{\bx} \\
=& -\int_{\Omega_{r_0}}\partial_{x_1}  w_1 \cdot \partial_{x_1}\tu_1 \dx{\bx} - \int_{\Omega\backslash\overline{\Omega_{r_0}}}\nabla  w_1 \cdot\nabla\tu_1 \dx{\bx}+k^2\int_{\Omega}w_1\tu_1 +w_1^2\dx{\bx},
\end{aligned}
\end{equation*}
where we have used the fact $\Delta\tu_1 = \partial_{x_1}^2\tu_1$ in $\Omega_{r_0}$. Applying H\"{o}lder’s inequality gives
\begin{equation*}
\begin{aligned}
\int_{\Omega} \abs{\nabla w_1}^2 \dx{\bx} \leqslant C \|\nabla w_1\|_{L^2(\Omega)}\left( \| \partial_{x_1}\tu_1\|_{L^2(\Omega_{r_0})} + \| \nabla \tu_1\|_{L^2(\Omega\backslash \overline{\Omega_{r_0})}}\right)\\
+k^2 \|w_1\|_{L^2(\Omega)}\left(  \| \tu_1\|_{L^2(\Omega)} +\|w_1\|_{L^2(\Omega)}\right).
\end{aligned}
\end{equation*}
Since $w_1 =0$ on $\partial \Omega$,
We then use the Poincar\'e inequality \eqref{eq:poincare2} to obtain
\begin{equation*}
 \|\nabla w_1\|_{L^2(\Omega)} \leqslant C  \left( \| \partial_{x_1}\tu_1\|_{L^2(\Omega_{r_0})} + \| \nabla \tu_1\|_{L^2(\Omega\backslash \overline{\Omega_{r_0}})} + k^2R \| \tu_1\|_{L^2(\Omega)} + k^2R^2\|\nabla w_1\|_{L^2(\Omega)} \right).
\end{equation*}
By the construction of $\tu_1$, its $H^2$ norm is bounded outside $\Omega_{r_0}$ and from   Lemma  \ref{lem:u1} we have
\begin{equation*}
\| \tu_1\|_{L^2(\Omega_{r_0})}^2 \leqslant   C (\abs{\lambda_1} + \abs{\lambda_2})^2 r_0 \left(\epsilon + \min\{r_1,r_2\}  \right).
\end{equation*}
In the quasi-static regime, we  obtain
\begin{equation*}
 \|\nabla w_1\|_{L^2(\Omega)} \leqslant C  \left(\| \partial_{x_1}\tu_1\|_{L^2(\Omega_{r_0})} +  \abs{\lambda_1} + \abs{\lambda_2}\right).
\end{equation*}
It remains to estimate  $\| \partial_{x_1}\tu_1\|_{L^2(\Omega_{r_0})}$. Since $x_2 +\epsilon+g_2(x_1) \leqslant \delta(x_1)$, we have
\begin{equation*}
\abs{ \partial_{x_1} \tu_1} =\abs{\frac{g_2'\delta - (x_2+\epsilon+g_2)\delta'}{\delta^2}(\lambda_1 -\lambda_2)} \leqslant 2 \frac{\delta'}{\delta}\abs{\lambda_1 -\lambda_2}.
\end{equation*}
By straight calculation, we have
\begin{equation*}
\| \partial_{x_1}\tu_1\|_{L^2(\Omega_{r_0})}^2  \leqslant 4\abs{\lambda_1 -\lambda_2}^2\int_{-r_0}^{r_0}\frac{\delta'^2}{\delta} \dx{x_1} \leqslant  C \frac{r_0}{\min\{r_1,r_2\}}\abs{\lambda_1 -\lambda_2}^2.
\end{equation*}
and this completes the proof.
\end{proof}

\begin{lemma}\label{lemma:graw}
Assume the quasi-static regime and $\min \{r_1,r_2\}\gg \epsilon$.
For any $\Omega_{\delta(\bz)} \subset \Omega_{r_0}$, we have
\begin{equation*}
\|\nabla w_1\|_{L^2(\Omega_{\delta(\bz)})} \leqslant C\left( \left(\frac{\delta(z_1)}{R} +k\delta(z_1) \right) \left( \abs{\lambda_1} + \abs{\lambda_2} \right) + \frac{\delta(z_1)}{\min\{r_1,r_2\}}\abs{\lambda_1-\lambda_2}  \right),
\end{equation*}
where $\Omega_{\delta(\bz)} := \Omega_{\delta(z_1)(\mathbf{z})}.$
\end{lemma}
\begin{proof}
For $0<t<s = \ocal (\sqrt{\min\{r_1,r_2\}\epsilon} +\abs{z_1})$, we define a cutoff smooth function $\eta(\bx)$ satisfying:
\begin{equation*}
\begin{cases}
\displaystyle \eta (x_1) = 1 \coma \abs{x_1-z_1} \leqslant t , \\
\displaystyle  0\leqslant \eta (x_1) \leqslant 1 \coma t < \abs{x_1-z_1} < s, \\
\displaystyle  \eta (x_1) =0 \coma \abs{x_1-z_1} \geqslant s .
\end{cases}
\end{equation*}
Multiplying the equation for $w_1$ by $\eta^2 w_1$ and integrating over $\Omega_{s(\bz)}$, we apply Green's formula to obtain
\begin{align*}
\int_{\Omega_{s(\bz)}} \eta^2 \abs{\nabla w_1}^2 \dx{\bx}&= \int_{\Omega_{s(\bz)}} k^2\eta^2w_1^2 +  \eta^2w_1 \left(\Delta \tu_1 + k^2\tu_1\right) - 2 w_1\nabla w_1 \cdot \eta\nabla \eta  \dx{\bx}.
\end{align*}
Then we use H\"older's inequality and get
\begin{align*}
\int_{\Omega_{s(\bz)}} \eta^2 \abs{\nabla w_1}^2 \dx{\bx} &\leqslant \frac{1}{2}\int_{\Omega_{s(\bz)}}2k^2\eta^2w_1^2  + \frac{\eta^2}{(s-t)^2}w_1^2 + \eta^2(s-t)^2 (\Delta \tu_1)^2  \medskip \\
& \hspace{50pt} +k^2\eta^2 (w_1^2+\tu_1^2) +\eta^2\abs{\nabla w_1}^2 +8\abs{\nabla \eta}^2w_1^2 \dx{\bx} \medskip\\
&\leqslant  \frac{1}{2}\int_{\Omega_{s(\bz)}} \frac{3k^2(s-t)^2 +9}{(s-t)^2} w_1^2 + (s-t)^2(\Delta \tu_1)^2 + k^2\tu_1^2 +\eta^2\abs{\nabla w_1}^2 \dx{\bx} \medskip
\end{align*}
which simplifies to
\begin{equation*}
\int_{\Omega_{s(\bz)}} \eta^2 \abs{\nabla w_1}^2 \dx{\bx}\leqslant \int_{\Omega_{s(\bz)}} \frac{3k^2(s-t)^2 +9}{(s-t)^2} w_1^2 +  (s-t)^2(\Delta \tu_1)^2 + k^2\tu_1^2 \dx{\bx} .
\end{equation*}
Since $s = \ocal (\sqrt{\min\{r_1,r_2\}\epsilon} +\abs{z_1})$, we have $\delta(x_1) = \ocal (\delta(z_1))$ for $\abs{x_1-z_1}\leqslant s$. By using the Poincar\'e inequality \eqref{eq:poincare1} we have
\begin{equation*}
\int_{\Omega_{s(\bz)}}  \abs{ w_1}^2 \dx{\bx} \leqslant c_0^2 \delta^2(z_1) \int_{\Omega_{s(\bz)}}  \abs{\nabla w_1}^2\dx{\bx} ,
\end{equation*}
where $c_0$ is some constant independent of $s$. Consequently,
\begin{align*}
&\int_{\Omega_{s(\bz)}} \eta^2 \abs{\nabla w_1}^2 \dx{\bx} \\
\leqslant&  (3+k^2(s-t)^2) \frac{ 3c_0^2\delta^2(z_1) }{(s-t)^2}  \int_{\Omega_{s(\bz)}}  \abs{\nabla w_1}^2 \dx{\bx} +  \int_{\Omega_{s(\bz)}} (s-t)^2(\Delta \tu_1)^2 + k^2\tu_1^2 \dx{\bx}.
\end{align*}
We further have by using  Lemma \ref{lem:u1},
\begin{align*}
 \int_{\Omega_{s(\bz)}} \eta^2 \abs{\nabla w_1}^2 \dx{\bx} &\leqslant (3+k^2(s-t)^2) \frac{ 3c_0^2\delta^2(z_1) }{(s-t)^2}  \int_{\Omega_{s(\bz)}}  \abs{\nabla w_1}^2 \dx{\bx} + \\
 &Cs\left(\frac{\abs{\lambda_1-\lambda_2}^2}{\min\{r_1^2,r_2^2\}} \frac{(s-t)^2}{\delta (z_1)} + k^2(\abs{\lambda_1} +\abs{\lambda_2})^2 \left(\delta(z_1) + \frac{s^2}{\min\{r_1,r_2\}}\right) \right) .
\end{align*}
where $C$ is some constant independent of $s,t$  and $\delta(z_1)$.

Let $s =t_{i+1}, t =t_i$ and $t_{i} = \delta(z_1) +2\sqrt{5}c_0\delta(z_1) i$. Then $(t_{i+1} - t_{i})^2 = 20c_0^2\delta^2(z_1)$ and we obtain
\begin{align*}
\displaystyle &\frac{\abs{\lambda_1-\lambda_2}^2}{\min\{r_1^2,r_2^2\}} \frac{(s-t)^2}{\delta (z_1)} + k^2(\abs{\lambda_1} +\abs{\lambda_2})^2 \left(\delta(z_1) + \frac{s^2}{\min\{r_1,r_2\}}\right) \medskip\\
\displaystyle =& C\delta(z_1)  \left( \frac{\abs{\lambda_1-\lambda_2}^2}{\min\{r_1^2,r_2^2\}}   +  k^2(\abs{\lambda_1} +\abs{\lambda_2})^2  \left(1+ \frac{\delta(z_1)}{\min\{r_1,r_2\}}(i+1)^2\right) \right), \medskip
\end{align*}
So we get the iterative inequality
\begin{align*}
 \int_{\Omega_{t_i(\bz)}}  \abs{\nabla w_1}^2  \dx{\bx} \leqslant &\frac{1}{2} \int_{\Omega_{t_{i+1}(\bz)}}  \abs{\nabla w_1}^2  \dx{\bx} + C \delta^2(z_1) \left(\frac{\abs{\lambda_1-\lambda_2}^2}{\min\{r_1^2,r_2^2\}}(i+1) \right.\\
 &\left.+ k^2(\abs{\lambda_1} +\abs{\lambda_2})^2 \left(1+\frac{\delta(z_1)}{\min\{r_1,r_2\}}(i+1)^2\right)(i+1) \right), 
\end{align*}
where we have used the quasi-static ansatz.
Iterating this inequality $m$ times yields
\begin{equation*}
\begin{aligned}
 \int_{\Omega_{\delta(\bz)}} \abs{\nabla w_1}^2  \dx{\bx} \leqslant &\left( \frac{1}{2}\right)^m  \int_{\Omega_{t_{m}(z_1)}} \abs{\nabla w_1}^2 \dx{\bx}  + C \delta^2(z_1)\sum_{i=1}^m \left( \frac{1}{2}\right)^i\left(\frac{\abs{\lambda_1-\lambda_2}^2}{\min\{r_1^2,r_2^2\}}(i +1) \right.\\
 &\left.+ k^2(\abs{\lambda_1} +\abs{\lambda_2})^2\delta^2(z_1) \left(1+\frac{\delta(z_1)}{\min\{r_1,r_2\}}(i+1)^2\right)(i+1) \right).
\end{aligned}
\end{equation*}
We let $m =\lfloor-2\log_2 (\delta(z_1)/R)\rfloor+1$ where $\lfloor x \rfloor$ is the floor function  denoting the greatest integer less than or equal to $x$. One can verify that $t_m =  \ocal (\sqrt{\min\{r_1,r_2\}\epsilon} +\abs{z_1}) $ and therefore
\begin{equation*}
 \int_{\Omega_{\delta(\bz)}} \abs{\nabla w_1}^2  \dx{\bx} \leqslant C \left( \left(\frac{\delta^2(z_1)}{R^2} +k^2\delta^2(z_1) \right)(\abs{\lambda_1} +\abs{\lambda_2})^2 + \frac{\delta^2(z_1)}{\min\{r_1^2,r_2^2\}}\abs{\lambda_1-\lambda_2}^2  \right).
\end{equation*}
Taking square roots completes the proof.
\end{proof}

\subsection{Proof of the upper bound}
With all the preliminary work above, we are in a position to establish the upper bound. 
\begin{proof}[Proof of the upper bound]
When $\min\{r_1,r_2\}\gg \epsilon$, we denote $\delta :=\delta(z_1)$ and introduce the scaled variables $y=(y_1,y_2)$ defined by
\begin{equation*}
x_1 - z_1 = \delta y_1 \coma x_2 = \delta y_2.
\end{equation*}
So the domain $\Omega_{\delta ({\bz})}$  is transformed into a nearly unit-size region
\begin{equation*}
Q_{1(\bz)} := \left\lbrace \by \in \bbr^2 |    - \frac{\epsilon}{\delta} - \frac{g_2(z_1+\delta y_1)}{\delta} < y_2 <   \frac{\epsilon}{\delta} + \frac{g_1(z_1+\delta y_1)}{\delta} , |y_1| < 1\right\rbrace.
\end{equation*}
Let
\begin{equation*}
W_1 (y_1,y_2)  := w_1(x_1,x_2) \coma \tilde{U}_1 (y_1,y_2)  := \tu_1(x_1,x_2).
\end{equation*}
and denote the top and bottom boundaries by $\Gamma_1$ and $\Gamma_2$, respectively.
In $\Omega_{r_0}$, the function $W_1$ satisfies
\begin{equation*}
\begin{cases}
\Delta W_1 +k^2\delta^2 W_1 = -\Delta \tilde{U}_1 -k^2 \delta^2 \tilde{U}_1 \quad &\textrm{in} \quad  \Omega_{r_0},  \\
W_1 = 0  &\textrm{on} \quad \Gamma_j ,j=1,2,
\end{cases}
\end{equation*}
By using Sobolev embedding $C^{1,1-2/p} \hookrightarrow H^{2,p}$ for $p>2$ and standard interior elliptic $L^p$ estimate, we obtain for $p=3$ 
\begin{equation*}
\begin{aligned}
\|\nabla W_1\|_{L^\infty (Q_{1/2(\bz)})} &\leqslant C\|W_1\|_{H^{2,p}(Q_{1/2(\bz)})} \leqslant C\left(  \|W_1\|_{L^p(Q_{1(\bz)})}  + \|\Delta \tilde{U}_1 + k^2 \delta^2\tilde{U}_1\|_{L^p(Q_{1(\bz)})}\right) \\
&\leqslant C \left( \|\nabla W\|_{L^2(Q_{1(\bz)})} +  \|\Delta \tilde{U}_1\|_{L^\infty(Q_{1(\bz)})} + k^2 \delta^2\|\tilde{U}_1\|_{L^\infty(Q_{1(\bz)})}\right).
\end{aligned}
\end{equation*}
Rescaling back, we have
\begin{equation*}
\|\nabla w_1\|_{L^\infty (\Omega_{\delta(\bz)/2})} \leqslant   C\delta^{-1} \left(\|\nabla w_1\|_{L^2(Q_{\delta(\bz)})} + \delta^2 \|\Delta \tu_1\|_{L^\infty(\Omega_\delta(\bz))} + k^2 \delta^2\|\tu_1\|_{L^\infty(\Omega_\delta(\bz))} \right).
\end{equation*}
Applying  Lemma \ref{lemma:graw} gives
\begin{align*}
\|\nabla w_1\|_{L^\infty (\Omega_{\delta(\bz)/2})}  \leqslant &C\delta^{-1}\left( \left(\frac{\delta}{R} +k\delta \right) \left( \abs{\lambda_1} + \abs{\lambda_2} \right) + \frac{\delta}{\min\{r_1,r_2\}}\abs{\lambda_1-\lambda_2}\right. \\
&\left. + \frac{\delta}{\min\{r_1,r_2\}}\abs{\lambda_1-\lambda_2}+  k^2\delta^2  \left( \abs{\lambda_1} + \abs{\lambda_2} \right)\right)\\
\leqslant  &C\delta^{-1}\left( \left(\frac{\delta}{R} +k\delta \right) \left( \abs{\lambda_1} + \abs{\lambda_2} \right) + \frac{\delta}{\min\{r_1,r_2\}}\abs{\lambda_1-\lambda_2} \right),
\end{align*}
and consequently
\begin{equation}\label{eq:nabla1}
\|\nabla w_1\|_{L^\infty (\Omega_{\delta(\bz)/2})} \leqslant   C\left( \left(\frac{1}{R} +k \right) \left( \abs{\lambda_1} + \abs{\lambda_2} \right) + \frac{1}{\min\{r_1,r_2\}}\abs{\lambda_1-\lambda_2} \right).
\end{equation}
Thus the gradient bound holds for any subdomain $\Omega_{\delta(\bz)/2} \subset \Omega_{r_0}$.
When $\bx \in \Omega\backslash \overline{\Omega_{r_0}}$, we have
  \begin{equation}\label{eq:nabla2}
 \begin{aligned}
 \|\nabla w_1(\bx)\|_{L^\infty ( \Omega\backslash \overline{\Omega_{r_0}})} &\leqslant  C  \left(\|\nabla w_1\|_{L^2(\Omega\backslash \overline{\Omega_{r_0/2}})} +  \|\Delta \tu_1 + k^2 \tu_1\|_{L^p(\Omega\backslash \overline{\Omega_{r_0/2}})} \right) \\
 &\leqslant C \left( \left(\abs{\lambda_1} +\abs{\lambda_2}\right) + \frac{1}{\min\{r_1,r_2\}}\abs{\lambda_1-\lambda_2}  \right),
\end{aligned}
\end{equation}
where we have used 
\begin{equation*}
\|\tu_1\|_{H^2(\Omega_{r_0}\backslash \overline{\Omega_{r_0/2}})} \leqslant C \left(\abs{\lambda_1} + \abs{\lambda_2} + \frac{1}{\min \{r_1+r_2\}+\epsilon} \abs{\lambda_1-\lambda_2}\right).
\end{equation*}
Since $u_1 = \tu_1 + w$,
we combined \eqref{eq:nabla1}, \eqref{eq:nabla2},\eqref{eq:tu1} and obtain
\begin{equation*}
\abs{\nabla u_1(\bx)} \leqslant C \left(\frac{1}{\delta(x_1)}\abs{\lambda_1-\lambda_2} + \left( \abs{\lambda_1} + \abs{\lambda_2} \right)   \right) \coma  \bx  \in \Omega.
\end{equation*}

Next, we consider the case $\min \{r_1,r_2\}= \ocal (\epsilon)$. If $\epsilon$ is of order one, Then $\|\tu_1\|_{C^2(\Omega)}$ is uniformly bounded. By Sobolev embedding and interior elliptic $L^p$ estimate, we have for any $V \subset \Omega$,
\begin{equation*}
\|\nabla w_1\|_{L^\infty(V)} \leqslant C \left(  \|\nabla w_1\|_{L^2(\Omega)} +  \|\Delta \tu_1 + k^2\tu_1\|_{L^p(\Omega)} \right) \leqslant C \left(\abs{\lambda_1} + \abs{\lambda_2} \right).
\end{equation*}
If $\epsilon \ll 1$ and $\bx \in  \Omega_{r_0}$, we redefine the cutoff function $\eta(\bx)$ by
\begin{equation*}
\begin{cases}
\displaystyle \eta (\bx) = 1 \coma \abs{x_j-z_j} \leqslant t , \\
\displaystyle  0\leqslant \eta (\bx) \leqslant 1 \coma t < \abs{x_j-z_j} < s, \\
\displaystyle  \eta (\bx) =0 \coma \abs{x_j-z_j} \geqslant s \coma j=1,2.
\end{cases}
\end{equation*}
Following a similar argument, we obtain gradient estimates analogous to \eqref{eq:nabla1} and \eqref{eq:nabla2}.


We claim that  $\nabla u_2$ is uniformly bounded due to the absence of the potential difference. To this end, we construct an auxiliary smooth function $\tu_2$ such that $u_2-\tu_2 \in H_0^1(\Omega)$ and $\|\tu_2\|_{C^{2}(\Omega)} \leqslant C$. 
Specifically, let $\tilde{R} : =2\left(r_1 + r_2 +\epsilon\right)$  and define a smooth cutoff function $\psi(\bx) $ by
\begin{equation*}
 \psi (\bx) = 0 \coma \abs{\bx} \leqslant \tilde{R}  \coma 0\leqslant \psi (\bx) \leqslant 1 \coma \tilde{R}< \abs{\bx} < \frac{\tilde{R} + R}{2}\coma 
  \psi (\bx) =1 \coma  \frac{\tilde{R} + R}{2} \leqslant \abs{\bx}\leqslant  R .
\end{equation*}
We transform the functions $\psi, f$ into polar coordinates form and denote $\tilde{\psi}(r,\theta): = \psi (\bx), \tilde{f}(\theta): = f (\bx)|_{\partial D_R},$
Since $\psi$ and $f$ are both smooth functions, the auxiliary function
$$\tilde{u}_2 (\bx) = \tilde{\psi}(r,\theta)\tilde{f}(\theta) \coma  \bx = (r\cos\theta,r\sin \theta),$$
 is also smooth  in $\Omega$ and satisfies $\left. \tilde{u}_2\right|_{\partial D_R} = f$.
Let $w_2 = u_2 -\tu_2$, we  have
\begin{equation*}
\begin{cases}
\Delta w_2 +k^2 w_2 = -\Delta \tu_2 -k^2 \tu_2 \quad &\textrm{in} \quad \Omega, \\
w_2 = 0  &\textrm{on} \quad \partial \Omega, j=1,2 .
\end{cases}
\end{equation*}
The fact $\tilde{u}(\bx)_2= 0$ when $\abs{\bx} <\tilde{R}$ implies no singularity between the two disks.
By following an argument analogous to $w_1$,
We can show that  $\nabla u_2$ is uniformly bounded and satisfies $\|\nabla u_2\|_{L^\infty(\Omega)} \leqslant C\|f\|_{C^2(\partial D_R)}$. This completes the proof.
\end{proof}

\bibliographystyle{abbrv}
\bibliography{referenceasy}{}

\end{document}